\documentclass[11pt, reqno]{amsart}

\usepackage[title,titletoc,toc]{appendix}
\usepackage{amsmath, amsthm}
\usepackage{amssymb}
\usepackage{color}
\usepackage{graphicx}
\usepackage{mathrsfs}

\usepackage[citecolor=blue]{hyperref}

\numberwithin{equation}{section}
\newtheorem{theorem}{Theorem}[section]

\newtheorem{proposition}[theorem]{Proposition}
\newtheorem{lemma}[theorem]{Lemma}

\DeclareMathOperator{\dist}{dist}
\DeclareMathOperator{\supp}{supp}
\DeclareMathOperator{\curl}{curl}
\newcommand{\calc}{{\mathcal C}}

\title[Increasing stability for conductivity coefficient]{Increasing stability for the conductivity and attenuation coefficients}

\author[Isakov]{Victor Isakov}
\address{Department of Mathematics and Statistics, Wichita State University, KS 67260-0033, USA.}
\curraddr{}
\email{victor.isakov@wichita.edu}
\thanks{The first and the second author were supported in part by the National Science Foundation.}

\author[Lai]{Ru-Yu Lai}
\address{Department of Mathematics, University of Washington, Seattle, WA 98195-4350,
USA.}
\curraddr{}
\email{rylai@math.washington.edu }

\author[Wang]{Jenn-Nan Wang}
\address{Institute of Applied Mathematical Sciences, NCTS, National Taiwan University, Taipei 106, Taiwan.}
\curraddr{}
\email{jnwang@math.ntu.edu.tw}
\thanks{The third author was supported in part by the MOST102-2115-M-002-009-MY3. }

\date{}

\begin{document}

\begin{abstract}
In this work we consider  stability of recovery of the conductivity and attenuation coefficients of the stationary Maxwell and Schr\"odinger equations from a complete set of
(Cauchy) boundary data. By using complex geometrical optics solutions we derive some bounds which can be viewed as an evidence of increasing stability in these inverse problems when frequency is growing.
\end{abstract}
\maketitle


\section{Introduction}

 The main goal of this paper is to demonstrate increasing stability of recovery of the conductivity coefficient from results of all possible electromagnetic boundary measurements when frequency of the stationary waves is growing. When the frequency is large the Maxwell system can not be reduced to a scalar conductivity equation, so we have to handle the full Maxwell system. We do it by reducing the first order system to a vectorial Schr\"odinger type equation  containing conductivity coefficient in matrix potential coefficient as in \cite{Ca}, \cite{OS}  and use ideas in \cite{I}, \cite{iw14} to derive increasing stability for potential coefficient involving variable attenuation. Use of the full system seems to be crucial, since increasing stability for the conductivity in the scalar equation does not seems to be true. We first consider a simpler case of scalar equation and extend the first and third author's result \cite{iw14} on the increasing stability estimate in the case of a constant attenuation to a variable one. Observe, that the method used in \cite{iw14} works only for  the constant attenuation. Extending this method to the case with variable attenuation is not an obvious routine, so the result for the scalar equation is of independent interest.

The problem of recovering the conductivity from boundary measurements has been well studied since the 1980s. The uniqueness issue was first settled by Sylvester and Uhlmann in \cite{SU1} where they constructed complex geometrical optics (CGO) solutions and proved global uniqueness of potential in the Schr\"odinger equation. A logarithmic stability estimate was obtained by Alessandrini \cite{A1}. Later, the log-type stability was shown to be optimal by Mandache \cite{M}. The logarithmic stability makes it impossible to design reconstruction algorithms with high resolution in practice since small errors in measurements will result in exponentially large errors in the reconstruction of the target material parameters. Nonetheless, in some cases, it has been observed numerically that the stability increases if one
increases the frequency in the equation. The study of the increasing stability phenomenon has attracted a lot of attention recently. There are several results \cite{I}, \cite{INUW}, \cite{iw14} and \cite{NUW} which
rigorously demonstrated the increasing stability behaviours in different settings.

For the Schr\"odinger equation without attenuation, the first author in \cite{I} derived some bounds in different ranges of frequency which can be viewed as an evidence of increasing stability behaviour. The idea in \cite{I} is to use complex- and real-valued geometrical optics solutions to control low and high frequency ranges, respectively. The proof was simplified in \cite{INUW} by using only complex-valued geometrical optics solutions. Recently, the first and third author showed in \cite{iw14} that the increasing stability also holds for the the same problem with constant attenuation. The proofs in \cite{iw14} use complex and bounded geometrical optics solutions.

Continuing the research in \cite{iw14} we show in this work even in the case of variable attenuation, the stability is improving when we increase the frequency. More precisely, the stability estimate will consist of two parts, one is the logarithmic part and the other one is the Lipschitz part (see Theorem~\ref{mainatt}). In the high frequency regime, the logarithmic part becomes weaker and the stability behaves like a Lipschitz continuity. Moreover, the constant of the Lipschitz part grows only polynomially in terms of the frequency.  We would like to point out that in \cite{NUW} an increasing stability phenomenon was proved for the acoustic equation in which the constant associated with the Lipschitz part grows exponentially in the frequency. In order to obtain polynomially growing constants as in \cite{I}, \cite{INUW} and \cite{iw14}, it seems like the smallness assumption on the attenuation is needed.

The main result of the paper is a proof of the increasing stability for the Maxwell equations in a conductive medium.
The first global uniqueness result for all three electromagnetic parameters of an isotropic medium was proved in \cite{OPS}. We refer the reader to the latest results in \cite{Ca0} where the uniqueness and stability (log-type) were established using local data.  Since our aim here is to demonstrate the increasing stability for the conductivity, we will assume that both the electric permittivity and the magnetic permeability are known constants. As in the previous results on the increasing stability, our proof here relies on CGO solutions. Here we will use CGO solutions constructed in \cite{Ca}.  Theorem~\ref{main}, shows that the logarithmic part becomes weaker and the Lipschitz part becomes dominant when the frequency is increasing. In order to derive a polynomially frequency-dependent constant in the Lipschitz part, we impose the smallness assumption on the conductivity. Of course, it is an interesting question whether one can remove the smallness assumption and still obtain a polynomially frequency-dependent constant in the Lipschitz part.

The  paper is structured as follows. Section 2 introduces the mathematical settings for the Schr\"odinger and the Maxwell equations and presents our main results. The proof of stability estimates for the Schr\"odinger equation is carried out in section 3. In section 4 we derive the stability estimate for the Maxwell equations.

\section{Preliminaries and main results}

Let $\Omega$ be a sub domain of the unit ball $B$ in $\mathbb{R}^3$, with Lipschitz boundary $\partial\Omega$. By $H^s(\Omega)$ we denote Sobolev spaces of functions on $\Omega$ with the standard norm $\|\cdot\|_{(s)}(\Omega)$.
These notations also hold for surfaces instead of $\Omega$
and for vector and matrix valued functions.

\subsection{Schr\"odinger equation with attenuation}

 We consider the Schr\"odinger equation with attenuation
\begin{align}\label{intro:sch}
     -\Delta u-\omega^2 u+ q u=0 \ \ \ \mbox{in $\Omega$}, \ \ q = i\omega\sigma +c,
\end{align}
 where $\sigma, c$ are real-valued bounded measurable functions in $\Omega$. Since the boundary value problem for \eqref{intro:sch} does not necessarily has a unique solution, for the study of the inverse problem we consider
 the Cauchy data set defined by
$$
    \calc(q)=\left\{(u,\partial_\nu u)|_{\partial\Omega} \in H^{1/2}(\partial\Omega)\times H^{-1/2}(\partial\Omega): \mbox{$u$ is a $H^1(\Omega)$-solution to (\ref{intro:sch})}    \right\}.
$$
Hereafter, $\nu$ is the unit outer normal vector to $\partial\Omega$. Let $q_j=i\omega \sigma_j+c_j$. To measure the distance between two Cauchy data sets, we define
$$
    \dist(\calc(q_1),\calc(q_2))=
    \sup_{j\neq k}\sup_{(f_j,g_j)\in \calc_{j}}  \inf_{(f_k,g_k)\in \calc_{k}}\frac{\|(f_j,g_j)-(f_k,g_k )\|_{(1/2, -1/2)}(\partial\Omega) }{\| (f_j,g_j) \|_{(1/2, -1/2)}(\partial\Omega)} ,
$$
where $(f_j,g_j)=(u_j,\partial_{\nu}u_j), j=1,2,$ are the Cauchy data for solutions $u_j$ to the equation
(\ref{intro:sch}) with $q=q_j$ and
$$
    \| (f,g) \|_{(1/2,-1/2)}(\partial\Omega)=
    \left( \|f\|^2_{(1/2)}(\partial\Omega)+ \|g\|^2_{(-1/2)}(\partial\Omega)  \right)^{1/2}.
$$
Denote
$$
  \epsilon= \dist(\calc_{1}, \calc_{2}),\ \  \mathcal{E}=-\log\epsilon.
$$

With these notations, the bound which indicates the increasing stability phenomenon for the Schr\"odinger equation with attenuation can now be stated as follows.
\begin{theorem}\label{mainatt}
Let $s>3/2$ and $q_j=i\omega\sigma_j+c_j$, $j=1,2$. Suppose that $\supp(q_1-q_2)\subset \Omega$ and $c_j, \sigma_j\in H^{2s}(\Omega)$ satisfy $\| c_j\|_{(2s)}(\Omega)
\leq M$ for some constant $M>0$. Let $ \omega>1,  \mathcal{E}>1$.

Then there exist  constants $m>0$ depending on $s, M$ and $C$ depending on $ s, \Omega, M$, such that if
$$
  \| \sigma_j\|_{(2s)}(\Omega)\leq m,
$$
then
\begin{align}\label{mainineq}
    \|q_1-q_2\|_{(-s)}(\Omega)\leq  C (\omega^2\epsilon+ \omega(\omega+ \mathcal{E}) ^{-(2s-3)/2}).
\end{align}
\end{theorem}

It is an immediate consequence of Theorem \ref{mainatt}  that if $\sigma_1=\sigma_2=b$ in $\Omega$ where $b$ is a fixed constant, $|b|\leq m$, then one gets a
 bound similar to \cite{iw14}.

When $\frac{5}{2}<s$ the logarithmic term is decreasing with respect to $\omega$, so the bound (\ref{mainineq}) implies improving (better than logarithmic) stability of recovery of  $q$ from complete boundary data. We discuss it in detail. 

Obviously,
$$
\omega^2\epsilon+ \omega(\omega+ \mathcal{E}) ^{-(2s-3)/2}
\leq \omega^2\epsilon +(\omega+ \mathcal{E}) ^{-(2s-5)/2}\leq \omega^2\epsilon +(\omega^2+ \mathcal{E}^2) ^{-(2s-5)/4}.
$$
The function $F(t)=\epsilon t +(t+ \mathcal{E}^2) ^{-(2s-5)/4}$ with the derivative
$F'(t)=\epsilon -(2s-5)/4(t+ \mathcal{E}^2) ^{-(2s-1)/4}$
is increasing on the interval $[1, +\infty)$ if
a) $(2s-5)/4(1+ \mathcal{E}^2) ^{-(2s-1)/4}\leq \epsilon$ and has minimum when $(2s-5)/4(t+ \mathcal{E}^2) ^{-(2s-1)/4}=\epsilon $ if
b) $\epsilon<(2s-5)/4(1+ \mathcal{E}^2) ^{-(2s-1)/4}$.
In the exceptional case a) the minimal value of $F$ on $[1,+\infty)$ is $\epsilon +(1+ \mathcal{E}^2) ^{-(2s-5)/4}\leq \epsilon + (4/(2s-5)\epsilon)^{(2s-5)/(2s-1)}$. In the case b)
the minimal value is  
$$
\epsilon((4/(2s-5)\epsilon)^{-4/(2s-1)}-\mathcal{E}^2)+
(4/(2s-5)\epsilon)^{(2s-5)/(2s-1)}\leq
$$
$$
\epsilon((4/(2s-5)\epsilon)^{-4/(2s-1)})+
(4/(2s-5)\epsilon)^{(2s-5)/(2s-1)}=
$$
$$
((4/(2s-5))^{-4/(2s-1)}+
(4/(2s-5))^{(2s-5)/(2s-1)})\epsilon^{(2s-5)/(2s-1)}.
$$
In the both cases, given error level $\epsilon$ one can choose frequency $\omega$ to guarantee at least H\"older stability for $q$ which is far better than logarithmic stability. Moreover, this analysis of the stability estimate (\ref{mainineq}) suggests that for given $\epsilon$ there is an optimal choice of $\omega$ for the best reconstruction of $q$.

\subsection{Maxwell equations}

 Let $E$ and $H$ denote the electric and magnetic vector fields in the medium with  the electric permittivity $\varepsilon$, $\mu$ is the magnetic permeability  $\mu$,
 and  the electric conductivity $\sigma$. The Maxwell equations at frequency $\omega$ are
\begin{align}\label{max2}
    \left\{
      \begin{array}{ll}
        \nabla \times  H+i\omega\gamma E=0, \\
        \nabla \times E -i\omega  \mu H=0,
        \gamma=\varepsilon+i\sigma/\omega.
      \end{array}
    \right.
\end{align}

We define the function space
$$
H(\curl;\Omega)=\{v\in H^0(\Omega):\nabla\times v\in H^0(\Omega)\},
$$
 with the norm
$$
   \|v\|_{H(\curl;\Omega)}=\|v\|_{(0)}(\Omega)+
   \|\nabla\times v\|_{(0)}(\Omega).
$$
For any $v\in H(\curl;\Omega)$ the tangential trace of $v$, $\nu\times v$, can be defined as an element of the  space $H^{-1/2}(\partial\Omega)$. Namely, for any $w^*\in H^{1}(\Omega)$,
$$
   \langle  \nu\times v, w \rangle=\int_\Omega(\nabla\times v)\cdot \overline w -\int_\Omega v\cdot (\overline{\nabla\times w}),
$$
 (see \cite{Ca}).
We define the space
\begin{align*}
   TH(\partial\Omega)=\{  w\in H^{-1/2}(\partial\Omega): \hbox{there exists $v\in H(\curl;\Omega),\nu\times v=w$}\}
\end{align*}
with the norm
$$
     \|w\|_{ TH(\partial\Omega)}=\inf\{\|v\|_{H(\curl;\Omega)}: \nu\times v=w\}.
$$
The Cauchy data set corresponding to (\ref{max2}) is defined by
\begin{align*}
\calc=\{(\nu\times E|_{\partial\Omega},&\nu\times H|_{\partial\Omega})\in (TH(\partial\Omega))^2: \\
&(E,H)\ \mbox{is}\ (H(\curl;\Omega))^2- \hbox{solution of (\ref{max2})} \}.
\end{align*}

In this work we assume that $\varepsilon=\mu=1$. Suppose that $\calc_1$ and $\calc_2$ are two Cauchy data of \eqref{max2} corresponding to conductivities $\sigma_1$ and $\sigma_2$, respectively. We measure the distance between two Cauchy data $\calc_1$, $\calc_2$ of the Maxwell equations as follows:
\begin{align*}
\dist(\calc_1,\calc_2)=\sup_{j\neq k}\sup_{   
  {\tiny \begin{array}{ll}
   (F(j),G(j))\in \calc_j,   
 \end{array}}
}
\inf_{(F(k),G(k))\in \calc_k}  \frac{\|(F(j),G(j))-(F(k),G(k) )\|_{(TH(\partial\Omega))^2}}{  \|F(j)\|_{TH(\partial\Omega)}  },
\end{align*}
where $F(j)=\nu\times E(j), G(j)=\nu\times H(j)$ on
$\partial\Omega$ and $E(j),H(j)$ is a $H(curl;\Omega)$
solution to (\ref{max2}) with $\sigma=\sigma_j$.
Likewise, we denote
$$
  \epsilon= \dist(\calc_{1}, \calc_{2}),\ \ \mathcal{E}=-\log\epsilon.
$$
Our main result is the following stability estimate for the Maxwell equations.

\begin{theorem}\label{main}
Let $s>3/2$ be an integer. Suppose that
$
\supp( \sigma_1- \sigma_2)\subset\Omega.
$
Assume that $\mathcal{E} >c $ (depending on $s,\Omega$) and $\omega> 1$.

 Then there exist  constants $m$ and $C$ depending on $s, \Omega$  such that if
\begin{align}\label{cond2}
    \|\sigma_j\|_{(2s+2)}(\Omega)  \leq m 
\end{align}
for $j=1,2$, then
\begin{align}\label{thm1}
       \|\sigma_1-\sigma_2\|_{(-s )}(\Omega) \leq  C\omega^{-1}(\omega^2+\mathcal{E}^2)^{3/2}\epsilon^{1/2}
             +C(\omega+\mathcal{E})^{-(2s-3)/2}+ 
             C (\omega + \mathcal{E})^{-1}.
\end{align}

\end{theorem}

As above, Theorem 2.2 implies increasing stability
for recovery of the conductivity coefficient.

\section{Stability estimates for Schr\"odinger equation}\label{sec3}

In this section we prove Theorem \ref{mainatt}. We will use some ideas in \cite{INUW}. $C$ will stay for a generic constant depending only on $s, \Omega, M$. First, following the arguments in \cite{FSU}, it is not hard to see that the following inequality holds.
\begin{proposition}
     Let $u_1$ and $u_2$ be the solutions of (\ref{intro:sch}) corresponding to the coefficients $(\sigma_1,c_1)$ and $(\sigma_2,c_2)$, respectively. Then
     \begin{equation}\label{alex}
         \left|\int_{\Omega}(q_1-q_2)u_1 u_2  \right|\leq  \left\| \left(u_1, \partial_\nu u_1\right)\right\|_{(1/2,-1/2)}(\partial\Omega)   \left\| \left(u_2, \partial_\nu u_2\right)\right\|_{(1/2,-1/2)}(\partial\Omega)\epsilon.
     \end{equation}
\end{proposition}

We use CGO solutions for (\ref{intro:sch}).

\begin{proposition}\label{cgosch}
 Let $\zeta\in \mathbb{C}^3$ satisfy $\zeta\cdot\zeta=\omega^2$. Let $s>3/2$. Then there exist constants $C_0$ and $C_1$, depending on $s$ and $\Omega$, such that if $ |\zeta|>C_0 \|q\|_{(2s)}(\Omega)$,
then there exists a solution
             $$
                u(x)=e^{i\zeta\cdot x}(1+\psi(x))
             $$
of the equation (\ref{intro:sch}) with
          $$
             \|\psi\|_{(2s)}(\Omega)\leq
             \frac{C_1}{|\zeta|}   \|q\|_{(2s)}(\Omega).
          $$
\end{proposition}
\begin{proof}

By extension theorems for Sobolev spaces there is an extension of $q$ from Lipschitz $\Omega$ onto the unit ball $B$ (denoted also  $q$) with $\|q\|_{(2s)}(B)\leq C(\Omega)
\|q\|_{(2s)}(\Omega)$. We can assume that $q$ is compactly supported in $B$.

We observe that since $e^{-i\zeta\cdot x}(-\Delta) (e^{i\zeta\cdot x}(1+\psi(x)))=(-\Delta-2i\zeta\cdot\nabla+\zeta\cdot\zeta)(1+\psi(x))$,
if $u=e^{i\zeta\cdot x}(1+\psi)$ is a solution of (\ref{intro:sch}) if and only if the reminder $\psi$ solves
\begin{align}\label{3:psi}
      (-\Delta-2i\zeta\cdot\nabla)\psi=-q(1+\psi)\quad\text{in}\quad\Omega.
\end{align}
As known \cite{SU1}, there are constants $C_0, C_1$, depending on $s$, such that if $C_0\|q\|_{(2s)}(B)<|\zeta|$, there exists a solution $\psi$ to (\ref{3:psi}) with
$$
\|\psi\|_{(2s)}(B)\leq \frac{C_1}{|\zeta|}\|q\|_{(2s)}(B).
$$
\end{proof}

 Let $\xi \in \mathbb{R}^3$. We select $e(j)\in \mathbb{R}^3$ satisfying $|e(1)|=|e(2)|=1$ and $e(j)\cdot\xi=e(1)\cdot e(2)=0$ for $j=1,2$. Let $R>0$. We choose
$$
\zeta(1) =-\frac{1}{2}\xi+
i \left(\frac{R^2}{2}\right)^{1/2} e(1)+
\left( \omega^2+\frac{R^2}{2}-
\frac{|\xi|^2}{4}\right)^{1/2}e(2),
$$
\begin{equation}
\label{38}
\zeta(2) =-\frac{1}{2}\xi-i\left(\frac{R^2}{2}\right)^{1/2} e(1)-\left( \omega^2+\frac{R^2}{2}-
\frac{|\xi|^2}{4}\right)^{1/2}e(2),
\end{equation}
provided
\[
\omega^2+\frac{R^2}{2}\ge\frac{|\xi|^2}{4}.
\]
It is clear that
$$
    \zeta(1)+\zeta(2)=-\xi,\ \ \ \
    \zeta(j)\cdot \zeta(j)=\omega^2,\ \ \
    |\zeta(j)|^2= R^2 +\omega^2,\ \ \ \mbox{for $j=1,2$},
$$
which implies that
$\sqrt{2}|\zeta(j)|\geq R+\omega.$
We observe that if
\begin{equation}
\label{dd2}
     R+\omega\geq \sqrt{2} C_0(m\omega+M),
\end{equation}
then one has
$$
    |\zeta(j)|
    \geq C_0(m\omega+M)>C_0 \|i\omega\sigma_j+c_j\|_{(2s)}(\Omega).
$$
Note that when
\begin{equation}
\label{dd3}
 R\ge\sqrt{2}C_0 M\quad\text{and}\quad \sqrt{2}C_0 m\le 1,
\end{equation}
the inequality \eqref{dd2} is clearly satisfied. So from Proposition \ref{cgosch}, there exist CGO solutions
\begin{equation}
\label{3.9}
     u_j(x)=e^{i\zeta(j)\cdot x}(1+\psi_j(x)),\ \        \|\psi_j\|_{(2s)}(\Omega)\leq \frac{C_1}{|\zeta(j)|}   \|i\omega\sigma_j+c_j\|_{(2s)}(\Omega).
\end{equation}

To derive the stability, we would like to estimate the Fourier transform of $q_1-q_2$. Before doing so, we need the following estimate on the Cauchy data of CGO solutions $u_j$, $j=1,2$, constructed above.
\begin{lemma}\label{3.12}
 If $u_j$ are the solutions (\ref{3.9}), then
\begin{align}\label{uH1}
\| (u_j,\partial_\nu u_j) \|_{(1/2,-1/2)}(\partial\Omega)
 \leq Ce^{\frac{R}{\sqrt{2}}}\sqrt{ R^2+\omega^2}.
 \end{align}
\end{lemma}

\begin{proof}
   Recall that $s>3/2$. Thus, by Sobolev embedding theorems
   $$
  \|\psi_j\|_{\infty}(\Omega)\leq
  C\|\psi_j\|_{(2s)}(\Omega)
             \leq \frac{CC_1}{|\zeta(j)|}
 \|i\omega\sigma_j+c_j\|_{(2s)}(\Omega)\leq
 \frac{CC_1}{C_0}\leq C
   $$
and
\[
|u_j(x)|=|e^{i\zeta(j)\cdot x}(1+\psi_j(x))|\le e^{|\Im \zeta(j)|}(1+|\psi_j(x)|)\leq C e^{\frac{R}{\sqrt{2}}}.
\]
 Moreover,
\[
|\nabla u_j(x)|\leq |i\zeta(j) e^{i\zeta(j)\cdot x} (1+\psi_j)|+|  e^{i\zeta(j)\cdot x}\nabla\psi_j |\leq Ce^{\frac{R}{\sqrt{2}}}|\zeta(j)|(1+|\nabla \psi(x)|),
\]
hence
$$
\|u_j\|_{(1)}(\Omega)\leq C e^{\frac{R}{\sqrt{2}}}
(1+|\zeta(j)|)(1+\|\psi_j\|_{(1)}(\Omega)
\leq C e^{\frac{R}{\sqrt{2}}}\sqrt{ R^2+\omega^2}
$$
and thus by Trace Theorems
$$
  \|(u_j,\partial_\nu u_j) \|_{(1/2,-1/2)}(\partial\Omega)
  \leq Ce^{\frac{R}{\sqrt{2}}}|\zeta(j)|=
  Ce^{\frac{R}{\sqrt{2}}}\sqrt{ R^2+\omega^2}.
$$
\end{proof}

We denote $\hat{f}$ the Fourier transform of the function $f$. Let $\tilde q = q_1-q_2$ . Lemma~\ref{3.12} is used to bound the low frequency of $\hat{\tilde q}$ as the following lemma shows.

\begin{lemma}
\label{lemma3.4}
Let $s>3/2$ and $\xi=r e$ with $ r\geq 0, |e|=1$. Assume that \eqref{dd3} is satisfied. There exists a constant $C$, such that
\begin{align}\label{qq}
   |\hat{\tilde q}(r e)|
   &\leq C\phi(R,\omega)  Q(\xi)   +Ce^{\sqrt{2}R}(R^2+\omega^2)\epsilon,
\end{align}
provided $\omega^2+R^2/2>r^2/4$.
Here
\begin{align*}
    Q(\xi)^2=\int \langle x\rangle^{-4s}  \int \langle -\eta+\xi-x\rangle ^{-2 s} |\hat{\tilde q}(\eta)|^2 d\eta dx
\end{align*}
and
$$
   \phi(R,\omega)= \frac{m\omega+M}{R+\omega}.
$$
\end{lemma}
\begin{proof}
Substituting CGO solutions (\ref{3.9}) into (\ref{alex}) yields
\begin{align*}
\left|\int_\Omega \tilde q(x)e^{-i\xi \cdot x}(1+\psi_1(x))(1+\psi_2(x)) dx  \right|\leq \notag\\
          \left\| \left(u_1, \partial_\nu u_1\right)\right\|_{(1/2,-1/2)}(\partial\Omega)\left\| \left(u_2, \partial_\nu u_2\right)\right\|_{(1/2,-1/2)}(\partial\Omega)\epsilon,
\end{align*}
which leads to
\begin{align}
\label{four11}
  &\left|\int_\Omega \tilde q e^{-i\xi \cdot x} dx  \right|   \notag\\
  &\leq  \left|\int  \tilde q(x) e^{-i\xi \cdot x}\chi(x)\Psi(x) dx  \right|+
   \left\| \left(u_1, \partial_\nu u_1\right)\right\|_{(1/2,-1/2)}(\partial\Omega)\left\| \left(u_2, \partial_\nu u_2\right)\right\|_{(1/2,-1/2)}(\partial\Omega)\epsilon,
\end{align}
where $\Psi(x)= \psi_1(x) + \psi_2(x)+\psi_1(x) \psi_2(x) $, $\chi\in C^\infty_0(\Omega)$, and $\chi=1$ in $\supp(\tilde q)$.

Since $\int  f\overline gdx=(2\pi)^{-3}\int \hat{f}\overline{\hat{g}}d\eta$ and $\widehat{fg}=(2\pi)^{-3}\hat{f}*\hat{g}$, one has
\begin{align*}
    & \int\tilde q(x) e^{-i\xi \cdot x}\chi(x)\Psi(x) dx   \\
    &=(2\pi)^{-3}\int \widehat{\tilde q}(\eta) (\widehat{e^{-i\xi \cdot x}\chi  \Psi})(-\eta)   d\eta=(2\pi)^{-6}\int \widehat{\tilde q}(\eta) (\widehat{e^{-i\xi \cdot x}\chi }*\widehat{\Psi})(-\eta)   d\eta.
\end{align*}
Denote $\chi_\xi=e^{-i\xi \cdot x}\chi $ and $\langle \xi \rangle=(1+|\xi|^2)^{1/2}$. By the H\"older's inequality, we obtain
\begin{align}\label{pre:four}
&|\int \hat{\tilde q}(\eta) (\widehat{e^{-i\xi \cdot x}\chi }*\widehat{\Psi})(-\eta)   d\eta| \notag\\
&\leq\int |\hat{\tilde q}(\eta) ||(\widehat{\chi_\xi}* \hat{\Psi}) (-\eta)|d\eta \notag\\
&\leq \int |\hat{\tilde q}(\eta)| |\int \widehat{\chi_\xi}(-\eta-x) \hat{\Psi}(x)dx|d\eta \notag\\
&\leq \int\int |\hat{\tilde q}(\eta)| | \widehat{\chi_\xi}(-\eta-x)|d\eta     | \hat{\Psi}(x)| dx  \notag\\
&\leq \left(\int \langle x\rangle^{-4s}  \left(\int |\hat{\tilde q}(\eta)| | \widehat{\chi_\xi}(-\eta-x)|     d\eta\right)^2        dx\right)^{1/2} \|\Psi\|_{(2s)}  \notag\\
&\leq \left(\int \langle x\rangle^{-4s}  \left(\int \langle -\eta+\xi-x\rangle ^{- s} |\hat{\tilde q}(\eta)|\langle -\eta+\xi-x\rangle ^{  s}  |\hat{\chi}(-\eta+\xi-x)|    d\eta\right)^2        dx\right)^{1/2} \|\Psi\|_{(2s)}\notag\\
&\leq \left(\int \langle x\rangle^{-4s}  \left(\int \langle -\eta+\xi-x\rangle ^{- 2s} |\hat{\tilde q}(\eta)|^2    d\eta\right)        dx\right)^{1/2}  \|\chi\|_{(s)}\|\Psi\|_{(2s)}.
\end{align}

From \eqref{3.9} and a priori assumptions, we have
\begin{equation}\label{339}
 \|\Psi\|_{(2s)}(\Omega) \leq C\frac{m\omega+M}{R+\omega}.
\end{equation}
Therefore, the estimate (\ref{qq}) follows from (\ref{uH1}), (\ref{four11}), (\ref{pre:four}) and (\ref{339}).
\end{proof}

The following lemma is an immediate consequence of Lemma \ref{lemma3.4}.
\begin{lemma}\label{lemma3.5}
 Let
\begin{equation}\label{ap2}
R^*>\sqrt{2}C_0M\quad\text{and}\quad \sqrt{2}C_0 m\le 1,
\end{equation}
where $C_0$ is the constant given in Proposition \ref{cgosch}.

If $0\leq r\leq \omega+\mathcal{R}^*$, then
\begin{align}\label{ff1}
|\hat{\tilde q}(r e)| &\leq  C \phi(R^*,\omega)  Q(\xi)  +C(\omega^2+R^{*2}) e^{\sqrt{2}R^*} \epsilon ;
\end{align}
if $r\geq \omega+R^*$, then
\begin{align}\label{ff2}
|\hat{\tilde q}(r e)|&\leq  C\phi(r,\omega) Q (\xi)    +C(\omega^2+r^2) e^{\sqrt{2}r} \epsilon.
\end{align}

\end{lemma}

\begin{proof}
It is enough to take $R=R^*$ when $0\leq r\leq \omega+R^*$ and take $R=r$ when $r\geq \omega+R^*$ in Lemma \ref{lemma3.4}.
\end{proof}

With the help of Lemma~\ref{lemma3.5} we can prove

\begin{lemma}
Let the conditions of Lemma~\ref{lemma3.5} be satisfied. Then for any $T\ge\omega+R^*$, we have
 $$
       \|\tilde q\|^2_{(-s)}(\Omega)
   \leq C\phi(R^*,\omega)^2\|\tilde q\|^2_{(-s)}(\Omega)
   $$
        \begin{equation}\label{pre:q}
   + C(\omega^2+T^2)^2( e^{2\sqrt{2}R^*}+ \chi(T) e^{2\sqrt{2}T} )\epsilon^2
       +C(m\omega+M)^2 T^{-(2s-3)},
 \end{equation}
where $\chi(T)\leq 1$ and $\chi(T)=0$ if $T=\omega+R^*$.
\end{lemma}
\begin{proof}
Using polar coordinates, we obtain
\begin{align}\label{F0}
    \|\tilde q\|^2_{(-s)}(\Omega)&=
    \int^\infty_0\int_{|e|=1}  |\hat{\tilde q}(re)|^2(1+r^2)^{-s} r^2 d e dr\notag\\
     &=\int^{\omega+R^*}_0\int_{|e|=1}  |\hat{\tilde q}(re)|^2(1+r^2)^{-s} r^2 d e dr\notag\\
     &\quad+\int_{\omega+R^*}^T\int_{|e|=1}  |\hat{\tilde q}(re)|^2(1+r^2)^{-s} r^2 d e dr\notag\\
     &\quad+\int_T^\infty\int_{|e|=1}  |\hat{\tilde q}(r e)|^2(1+r^2)^{-s} r^2 d e dr\notag\\
     &=:I_1+I_2+I_3.
\end{align}

We begin with bounding $I_3$. Since $\supp \tilde q \subset \Omega$ by the H\"older's inequality
$|\hat{\tilde q}(\xi)|\leq C\|\tilde q\|_{(0)}(\Omega)$ and so
\begin{align}\label{F1}
I_3& \leq
C \|\tilde q\|^2_{(0)}(\Omega)\int_T^\infty\int_{|e|=1}   (1+r^2)^{-s}  r^2 de dr\notag\\
     &\leq C \|\tilde q\|^2_{(0)}(\Omega)T^{-(2s-3)}
     \leq C(\omega  m+M)^2 T^{-(2s-3)}.
\end{align}

Before evaluating $I_1$ and $I_2$ terms, we need the following estimate.
Let $A=\{\eta;\ |-\eta+\xi-x|\leq |\eta|/2\}.$ By direct computation, we have
\begin{align*}
&  \int \langle \xi\rangle^{-2s}   \int \langle x\rangle^{-4s}  \left(\int_{A^c} \langle -\eta+\xi-x\rangle ^{-2 s} |\hat{\tilde q}(\eta)| ^2 d\eta \right)        dx  d\xi \\
& \leq  C  \int \langle \xi\rangle^{-2s}   \int \langle x\rangle^{-4s}  \left(\int_{A^c} \langle \eta \rangle ^{-2 s} |\hat{\tilde q}(\eta)| ^2 d\eta \right)        dx  d\xi\\\
&\leq C \|\tilde q\|^2_{(-s)},
\end{align*}
with $C$ depends on $s$.
On the other hand, by using the fact that $A\subset \{\eta;\  \frac{2}{3}|\xi-x|\leq |\eta|\leq 2|\xi-x|\}$ and $\langle x\rangle^{-2s}\langle \xi\rangle^{-2s} \leq 2^s\langle x-\xi\rangle^{-2s}$, one has
\begin{align*}
&  \int \langle \xi\rangle^{-2s}   \int \langle x\rangle^{-4s}  \left(\int_{A } \langle -\eta+\xi-x\rangle ^{-2 s} |\hat{\tilde q}(\eta)| ^2 d\eta \right)        dx  d\xi  \\
&\leq   \int \langle \xi\rangle^{-2s}   \int \langle x\rangle^{-4s}   \int_{\frac{2}{3}|\xi-x|\leq |\eta|\leq 2|\xi-x|} \langle -\eta+\xi-x\rangle ^{-2 s} |\hat{\tilde q}(\eta)| ^2 d\eta        dx  d\xi \\
&\leq  C \int \langle x\rangle^{- 2s}   \int \int_{\frac{2}{3}|\xi-x|\leq |\eta|\leq 2|\xi-x|} \langle x-\xi\rangle^{- 2s} \langle -\eta+\xi-x\rangle ^{-2 s} |\hat{\tilde q}(\eta)| ^2 d\eta        d\xi dx     \\
& \leq C  \int \langle x\rangle^{- 2s}   \int \int_{\frac{2}{3}|\xi-x|\leq |\eta|\leq 2|\xi-x|} \langle \eta \rangle^{- 2s} \langle -\eta+\xi-x\rangle ^{-2 s} |\hat{\tilde q}(\eta)| ^2 d\eta        d\xi dx    \\
& \leq C  \int \left(   \int \int  \langle x\rangle^{- 2s}  \langle -\eta+\xi-x\rangle ^{-2 s} d\xi dx \right) |\hat{\tilde q}(\eta)| ^2 \langle \eta \rangle^{- 2s} d\eta  \\
&\leq C\|\tilde q\|^2_{(-s)}.
\end{align*}

Therefore, we can deduce that
\begin{align}\label{newQ}
    \int \langle\xi\rangle^{-2s} Q(\xi)^2 d\xi\leq C\|\tilde q\|_{(-s)}^2.
\end{align}

Next, by using (\ref{ff1}) and (\ref{newQ}), we yield
\begin{align}\label{F2}
I_1&\leq \int^{\omega+R^*}_0 C \big(\phi(R^*,\omega)
  Q (\xi)  \notag+(\omega^2+R^{*2}) e^{\sqrt{2}R^*}\epsilon\big)^2 (1+r^2)^{-s} r^2  dr\notag\\
&\leq C(\phi(R^*,\omega)^2\|\tilde q\|^2_{(-s)}(\Omega) +(\omega^2+R^{*2})^2 e^{\sqrt{2}R^*} \epsilon^2)\int^{\omega+R^*}_0 (1+r^2)^{-s} r^2  dr \notag\\
&\leq C\phi(R^*,\omega)^2\|\tilde q\|^2_{(-s)}(\Omega)+C(\omega^2+R^{*2})^2 e^{\sqrt{2}R*} \epsilon^2,
\end{align}
since due to $2s>3$  we have
$$
\int_0^{\omega+R^*}  (1+r^2)^{-s} r^2  dr\leq
\int^\infty_0  (1+r^2)^{-s} r^2   dr=C.
$$

In the same way, the estimate (\ref{ff2}) implies that
\begin{align}\label{F3}
 I_2&\leq
 C\big( \phi(R^*,\omega)^2\|\tilde q\|^2_{(-s)}(\Omega)+(\omega^2+T^2)^2 e^{\sqrt{2}T}\epsilon^2\big) \int_{\omega+R^*}^T (1+r^2)^{-s} r^2  dr\notag\\
     &\leq  C\phi(R^*,\omega)^2\|\tilde q\|^2_{(-s)}(\Omega) +C(\omega^2+T^2)^2 e^{2\sqrt{2}T}\epsilon^2.
\end{align}

Combining (\ref{F0})-(\ref{F3}) and using that $I_2=0$ when
$T=\omega+ R^*$, we complete the proof.

\end{proof}

We are now ready to prove Theorem~\ref{mainatt}.

\medskip\noindent
\emph{Proof of Theorem \ref{mainatt}.} We will use the estimate \eqref{pre:q}. Obviously,
\[
\phi(R^*,\omega)=\frac{m\omega+M}{ R^*+\omega}\le m+\frac{M}{ R^\ast}.
\]
Therefore, we can choose $m<1/(\sqrt{2}C_0)$ and $ R^\ast>\sqrt{2}C_0 M$ so that
\[
C\phi(R^\ast,\omega)^2<\frac 12.
\]
The choice of $m$ depends only on $s,\Omega$ and of $ R^\ast$ only depends on $s,\Omega, M$. Thus, the first term on the right hand side of \eqref{pre:q} can be absorbed by the left side, and we obtain
\begin{align}
\label{lqq}
      \|\tilde q\|^2_{(-s)}(\Omega)&\leq
         C(\omega^2+T^2)^2(e^{2\sqrt{2}R^*}+
        \chi(T)e^{2\sqrt{2}T}) \epsilon^2
       +C\omega^2T^{-(2s-3)}.
 \end{align}

We consider the following two cases:
\begin{align*}
    \text{(i)}\ \omega+R^*\leq \frac{1}{2}\mathcal{E}\ \ \hbox{and}\ \ \text{(ii)}\ \omega+\mathcal{R}^*\geq \frac{1}{2}\mathcal{E}.
\end{align*}

In the  case (i) we choose
$T=\frac{\mathcal{E}}{2}.$
We want to show that there exists $C_2$ such that
\begin{equation}\label{res1}
(\omega^2+T^2)^2 e^{2\sqrt{2}T} \epsilon^2\leq C_2\omega^2(\omega+\mathcal{E})^{-(2s-3)}
\end{equation}
and
\begin{equation}\label{res2}
\omega^2T^{-(2s-3)}\leq C_2\omega^2(\omega+\mathcal{E})^{-(2s-3)}.
\end{equation}
It is clear that \eqref{mainineq} is a consequence of (\ref{lqq}), (\ref{res1}), and (\ref{res2}).

Due to our choice of $T$, (\ref{res1}) is satisfied if
$$
(\omega+\mathcal{E})^4\epsilon^{2-\sqrt{2}}
\leq C(\omega+\mathcal{E})^{-2s+3},
$$
or
$$
(\omega+\mathcal{E})^{4+2s-3}\epsilon^{2-\sqrt{2}}
\leq C,
$$
which follows from the elementary inequality
$(\mathcal{E})^{4+2s-3}\epsilon^{2-\sqrt{2}}
\leq C.$

Note that (\ref{res2}) is equivalent to
\begin{align} \label{res22}
     C_2^{-1/(2s-3)}(\omega+\mathcal{E}) \leq T= \frac{\mathcal{E}}{2} .
\end{align}
Obviously,
$$
    \omega+\mathcal{E} \leq \omega+R^*+\mathcal{E}\leq \frac{3\mathcal{E}}{2},
$$
according to (i). Now we see that (\ref{res22}) holds whenever
$ C_2^{-1/(2s-3)}\leq 1/3$.

Next we consider case (ii). We choose $T=\omega+R^*.$
Then from (\ref{lqq}) we have
\begin{equation*}
            \|\tilde q\|^2_{H(-s)}(\Omega)
            \leq C(\omega^4+R^{*4}) e^{2\sqrt{2}R^*}  \epsilon^2 +C\omega^2(\omega+R^*)^{-(2s-3)}.
\end{equation*}
Since $R^*<C$ and in the case (ii)
$\omega+R^*\geq \frac{\omega}{2} +\frac{1}{4}\mathcal{E},
$ the estimate (\ref{mainineq}) follows.

The proof now is complete.

\section{Stability estimates for Maxwell equations}

 In this section we will prove Theorem~\ref{main}. We first discuss a reduction that transforms the Maxwell equations to the vectorial Schr\"odinger equation by adapting Section~1 in \cite{COS} to the special case of $\varepsilon=\mu=1$. We then use  arguments similar to the previous section to derive a bound of the difference of the conductivities.

\subsection{Reduction to the Schr\"odinger equation}
Let $\Omega$ be a bounded domain in $\mathbb{R}^3$ with smooth boundary. We consider the time-harmonic Maxwell equations (\ref{max2}) with $\varepsilon=\mu=1$,
\begin{align}\label{r:max}
    \left\{
      \begin{array}{ll}
        \nabla \times  H+i\omega\gamma E=0, \\
        \nabla \times  E -i\omega  H=0,
      \end{array}
    \right.
\end{align}
where $\gamma=1+i\sigma/\omega$, $\sigma\in C^2(\overline\Omega)$ and $\sigma>0$ in $\overline\Omega$.
From (\ref{r:max}), one has the following compatibility conditions for $E$ and $H$
\begin{align}\label{2:max}
    \left\{
      \begin{array}{ll}
       \nabla \cdot(\gamma E)=0, \\
        \nabla\cdot H=0.
      \end{array}
    \right.
\end{align}
Let $\alpha=\log\gamma$ (the principal branch). Combining (\ref{r:max}) and (\ref{2:max}) gives the following eight equations
\begin{align*}
    \left\{
      \begin{array}{ll}
       \nabla\cdot E+\nabla\alpha\cdot E=0, \\
      \nabla \times  E -i\omega  H=0, \\
       \nabla \cdot H=0, \\
        \nabla \times  H+i\omega\gamma E=0,\\
      \end{array}
    \right.
\end{align*}
which leads to the $8\times 8$ system
$$
      \left[\left(
              \begin{array}{cccc}
                \ast & 0 & \ast & D\cdot \\
                \ast & 0 & \ast & -D\times \\
                \ast & D\cdot & \ast & 0 \\
                \ast & D\times & \ast & 0 \\
              \end{array}
            \right)+
     \left(
       \begin{array}{cccc}
         \ast & 0 & \ast & D\alpha\cdot \\
         \ast & \omega I_3 & \ast & 0 \\
         \ast & 0 & \ast & 0 \\
         \ast & 0 & \ast & \omega\gamma I_3 \\
       \end{array}
     \right)
\right]\left(
                     \begin{array}{c}
                       0 \\
                       H \\
                       0 \\
                       E \\
                     \end{array}
                   \right)=0.
$$
Here $I_j$ is the $(j\times j)$-identity matrix,
$$
    D\cdot =-i(\partial_1\ \partial_2\ \partial_3),\ \ D=-i(\partial_1\ \partial_2\ \partial_3)^t
$$
and
$$
    D\times=-i\left(
              \begin{array}{ccc}
                  0& -\partial_3 & \partial_2 \\
                \partial_3 &  0 & -\partial_1 \\
                -\partial_2 &\partial_1 &  0 \\
              \end{array}
            \right).
$$
$\ast$ means that we obtain the same equations regardless of values of $\ast$.

Define the operators
$$
  P_+(D)=\left(
           \begin{array}{cc}
             0 & D\cdot \\
             D & D\times\\
           \end{array}
         \right),\ \ P_-(D)=\left(
           \begin{array}{cc}
             0 & D\cdot \\
             D & -D\times \\
           \end{array}
         \right)
$$
acting on 4-vectors. It is easy to check that
$$
   P_+(D)P_-(D)=P_-(D)P_+(D)=-\Delta I_4
$$
and
$$
    P_+(D)^*=P_-(D),\ \ P_-(D)^*=P_+(D).
$$
Also, we can see that the $8\times 8$ operator
$$
    P=\left(
        \begin{array}{cc}
          0 & P_-(D) \\
          P_+(D) & 0 \\
        \end{array}
      \right)
$$
is a self-adjoint, i.e., $P^*=P$.
For $\zeta\in{\mathbb C}^3$, we set
$$
    P(\zeta)=-i\left(
        \begin{array}{cc}
          0 & P_-(\zeta) \\
          P_+(\zeta) & 0 \\
        \end{array}
      \right).
$$

Let $X=(\Phi_1,\ H^t, \Phi_2,\ E^t)^t$ be 8-vectors, where $\Phi_1$ and $\Phi_2$ are additional scalar fields. The augmented Maxwell's system is
$$
     (P+V)X=0\ \ \hbox{in $\Omega$},
$$
where
$$
      V=\left(
          \begin{array}{cccc}
            \omega & 0 & 0 & D\alpha\cdot \\
            0 & \omega I_3 & D\alpha & 0 \\
            0 & 0 & \omega\gamma & 0 \\
            0 & 0 & 0 & \omega\gamma I_3 \\
          \end{array}
        \right).
$$
Note that when $\Phi_j=0$ the solution $X$ corresponds to the solution of the original Maxwell's system in $\Omega$. Now we want to rescale the augmented system. Let
$$
     Y=\left(
         \begin{array}{cc}
           I_4 & 0 \\
           0 & \gamma^{1/2} I_4\\
         \end{array}
       \right)X,
$$
then
$$
      (P+V)X= \left(
         \begin{array}{cc}
           \gamma^{-1/2} I_4& 0 \\
           0 & I_4\\
         \end{array}
       \right)  (P+W)Y,
$$
where
\begin{align*}
W=\kappa I_8+\frac{1}{2}
\left(
\begin{tabular}{cc|cc}
  0 & 0 & 0 & $D\alpha\cdot$ \\
  0& 0&   $D\alpha $& $D\times \alpha$\\
  \hline
  0 & 0 &0 &0 \\
    0 & 0 &0 &0 \\
\end{tabular}
       \right)
\end{align*}
with $\kappa=\omega\gamma^{1/2}$. Hence
$$
(P+V)X=0\quad\text{if and only if}\quad   (P+W)Y=0.
$$

The advantage of rescaling is that no first order terms appear in (\ref{P1})-(\ref{P3}). The proof of next lemma can be found in \cite{COS}.
\begin{lemma}\label{lemma4.1}
\begin{align}
   (P+W)(P-W^t)&=-\Delta I_8+Q,\label{P1}\\
   (P-W^t)(P+W)&=-\Delta I_8+Q(1),\\
   (P+W^*)(P-\overline W)&=-\Delta I_8+Q(2)\label{P3},
\end{align}
where the matrix potentials are given by
\begin{align}\label{qqq}
    Q =\frac{1}{2}
\left(
\begin{tabular}{cc|cc}
  $\Delta\alpha$ &  &  & \\
  & $2\nabla^2\alpha-\Delta\alpha I_3$  &   &\\
  \hline
  & & & \\
     &  & & \\
\end{tabular}
       \right)
- \kappa^2 I_8
  -\left(
\begin{tabular}{c|c}
   $\frac{1}{4}(D\alpha\cdot D\alpha)I_4$  &
$
  \begin{array}{cc}
     & 2D\kappa \cdot \\
    2D\kappa &  \\
  \end{array}
$ \\
  \hline
$ \begin{array}{cc}
     & 2D\kappa\cdot \\
    2D\kappa &  \\
  \end{array}
$  &  \\
\end{tabular}
       \right),
\end{align}
$$
    Q(1)=-\frac{1}{2}
\left(
\begin{tabular}{cc|cc}
  & & & \\
     &  & & \\
  \hline
 &  & $\Delta\alpha$ & \\
  &  &   & $2\nabla^2\alpha-\Delta\alpha I_3$\\
\end{tabular}
       \right)
-\kappa^2 I_8
-
\left(
\begin{tabular}{c|c}
 &$
  \begin{array}{cc}
     & \\
     & 2D\kappa\times \\
  \end{array}
$
\\
  \hline
$ \begin{array}{cc}
     & \\
    &  -2D\kappa\times\\
  \end{array}
$  &  $\frac{1}{4}(D\alpha\cdot D\alpha)I_4$\\
\end{tabular}
       \right),
$$
and
\begin{align*}
    Q(2)=-\frac{1}{2}
\left(
\begin{tabular}{cc|cc}
  & & & \\
     &  & & \\
  \hline
 &  & $\Delta\overline\alpha$ & \\
  &  &   & $2\nabla^2\overline\alpha-\Delta\overline\alpha I_3$\\
\end{tabular}
       \right)
-\overline\kappa^2 I_8
 -\left(
\begin{tabular}{c|c}
 &$
  \begin{array}{cc}
     & \\
     & -2D\overline\kappa\times\\
  \end{array}
$
\\
  \hline
$ \begin{array}{cc}
     & \\
    &  2D\overline\kappa\times\\
  \end{array}
$  &  $\frac{1}{4}(D\overline \alpha\cdot D\overline\alpha)I_4$\\
\end{tabular}
       \right).
\end{align*}
Here $\nabla^2 f=(\partial^2_{ij}f)$ is the Hessian of $f$
and  only non zero elements are shown in  $8\times 8$ matrices $Q$, $Q(1)$, and $Q(2)$. Moreover, $W^t$ denotes the transpose of $W$ and $W^*=\overline{W}^t$.
\end{lemma}

\subsection{Construction of CGO solutions}
We recall that  $\overline\Omega\subset B$.
We extend $\sigma$ to a function in $H^{2s+2}(\mathbb{R}^3)$, still denoted by $\sigma$, so that $\supp\sigma$ in $ B$. Therefore, $\omega^2I_8+Q$, $\omega^2I_8+Q(1)$ and $\omega^2I_8+Q(2)$ are compactly supported in $B$.

Let $Y=(f^1,\ \ (u^1)^t, \ \ f^2,\ \ (u^2)^t)^t\in H^{2s}(B)$.
We recall the following construction of CGO solutions for Maxwell equations (\ref{r:max}) in \cite{Ca}.
\begin{proposition}{\rm\cite[Proposition~9]{Ca}}\label{cgo:max}
Let $s>3/2 $ be an integer. Assume that $a, b,\zeta\in\mathbb{C}^3$ and $\zeta\cdot\zeta=\omega^2$. Then there exists a constant $C_0$ depending on $s$, such that if
$$
    |\zeta|>C_0\left(\sum_{j=1,2} \|\omega^2+q_j\|_{(2s)}(B )+\|\omega^2 I_8+Q\|_{(2s)}(B )\right),
$$
where
$$
  q_1=-\kappa^2,\ \ q_2=-\frac{1}{2}\Delta\alpha-\kappa^2-\frac{1}{4}(D\alpha\cdot D\alpha),
$$
then there exists a solution
$$
   Z=e^{i\zeta\cdot x}(A+\Psi)
$$
of $(-\Delta I_8+Q)Z=0$ with
$$
    A=\frac{1}{|\zeta|}
\left(
    \begin{tabular}{c}
     $\zeta\cdot a$\\
     $\omega b$\\
   \hline
      $\zeta\cdot b$\\
      $\omega a$
    \end{tabular}
\right),
$$
satisfying
$$
      \|\Psi\|_{(2s)} (\Omega)\leq \frac{C_0}{|\zeta|} |A| \|\omega^2 I_8+Q\|_{(2s)}(B).
$$
Moreover, $Y=(P-W^t)Z$ is a solution of $(P+W)Y=0$ and
$$
    Y=
(0,\ \ H^t, \ \ 0,\ \ \gamma^{1/2}E^t)^t
$$
where $E,H$ solve (\ref{r:max}).
\end{proposition}
\begin{proof}
The proposition can be proved following  \cite[Proposition~9]{Ca}. The only difference is that we use the estimate from $H_{\delta+1}^{2s}$ to $H_{\delta}^{2s}$ with $-1<\delta<0$, that is,
\begin{align}\label{funest}
    \|G_\zeta f\|_{H_\delta^{2s}}\leq \frac{C(\delta)}{|\zeta|}\|f\|_{H_{\delta+1}^{2s}},
\end{align}
instead of \cite[(22)]{Ca} where the estimate is from $L^2_{\delta+1}$ to $L^2_{\delta}$, where $G_\zeta$ is the convolution operator with the fundamental solution of $(-\Delta-2i\zeta\cdot\nabla)$.   This estimate (\ref{funest}) holds due to the easy fact $(-\Delta-2i\zeta\cdot \nabla)$ and $(I-\Delta)^{2s}$ commute. Due to the fact that $\omega^2I_8+Q$ is compactly supported in $B$, one can replace the $H_{\delta+1}^{2s}$ norm on the right hand side of \eqref{funest} by $H^{2s}(B)$. Recall that $H^{2s}(B)$ is a Banach algebra for $s>3/4$. The rest of the proof follows  \cite[Proposition~9]{Ca}.
\end{proof}

We also need CGO solutions for the rescaled system. These solutions can be constructed as in \cite[Lemma~10, Proposition~11]{Ca}.
\begin{proposition}{\rm\cite[Lemma~10, Proposition~11]{Ca}}\label{cgo:res}
Let $a^*, b^*,\zeta\in\mathbb{C}^3$ and $\zeta\cdot\zeta=\omega^2$.
  Then there exists a constant $C_0$ depending on $s$, such that if
$$
    |\zeta|>C_0 \|\omega^2 I_8+Q(2)\|_{(2s)}(B ),
$$
then there is a solution
$$
Y^*=e^{i\zeta\cdot x}(A^*+\Psi^*)
$$
of $(P+W^*)Y^*=0$, where
$$
    A^*=\frac{1}{|\zeta|}
\left(
    \begin{tabular}{c}
     $\zeta\cdot a^*$\\
     $-\zeta\times a^*$\\
   \hline
      $\zeta\cdot b^*$\\
      $\zeta\times b^*$
    \end{tabular}
\right)
$$
and
$\Psi^*=P\Psi_*+iP(\zeta) \Psi_*-\overline{W} A_*-\overline{W}\Psi_*$ with
$A_*=\frac{1}{|\zeta|}(0,\ \ b^{*t}, \  \ 0, \ \ a^{*t})^t$,
\begin{align}\label{4:R1}
      \|\Psi_*\|_{(2s)} (\Omega)\leq \frac{C_0}{|\zeta|} |
      A_*| \|\omega^2 I_8+Q(2)\|_{(2s)}(B)
\end{align}
and
\begin{align}\label{4:R2}
      \|P\Psi_*\|_{(2s)} (\Omega)\leq C_0 ( |A_*|+\|\Psi_*\|_{(2s)} (\Omega)) \|\omega^2 I_8+Q(2)\|_{(2s)}(B)
\end{align}
 Moreover,
\begin{equation}\label{0809}
      \|\Psi^*\|_{(2s)}(\Omega)\leq C_0|A_*|\left(1 + \frac{ \|\sigma \|_{(2s+2)}(B)\omega }{|\zeta|}\right) \left(\|\omega^2 I_8+Q(2)\|_{(2s)}(B) +
      \|W\|_{(2s)}(B)\right).
\end{equation}
\end{proposition}

\subsection{Stability estimates}
Suppose that $\varepsilon=\mu=1$, $ \|\sigma_j \|_{(2s+2)}(B)\leq m<1$ and $\supp(\sigma_1-\sigma_2)\subset\Omega$. Hence we have $\supp(\gamma_1-\gamma_2)\subset\Omega$.
We first prove an inequality that connects the unknowns and the boundary measurements.

In the proofs $C$ denote generic constants depending
on $s, \Omega$.

\begin{theorem}\label{caroq}
Assume that $\sigma_1$ and $\sigma_2$ belong to $H^{2s+2}(\Omega)$ and $\supp(\sigma_1-\sigma_2)\subset\Omega$. There exists a constant $C$ dependent on $\Omega,s$ such that, for any $Z_1\in H^{2s}(\Omega)$ satisfying $Y_1=(P-W_1^t)Z_1 =(0, \ \ H^t_1, \ \ 0, \ \ \gamma_1^{1/2}E^t_1)^t$
with $(E_1,H_1)$ solution to (\ref{r:max}) in $\Omega$ with coefficient $\sigma_1$,
and any $H^{2s}(\Omega)$ solution $Y_2=(f^1, \ \ (u^1)^t, \  \ f^2,\ \ (u^2)^t)^t$ of $(P+W^*_2)Y_2=0$, one has
\begin{align}\label{id}
     |((Q_1-Q_2)Z_1,Y_2)_\Omega|&\leq C\dist({\mathcal C}_{1},{\mathcal C}_{2})\|Y_1\|_{(1)}(\Omega)\|Y_2\|_{(1)}(\Omega).
\end{align}
\end{theorem}

Note that $Q_j$ is defined in \eqref{qqq} corresponding to $\sigma_j$ for $j=1,2$. The matrix-valued functions $W_{1}, W_2$ are defined similarly.

\begin{proof}
Using $\supp(\sigma_1-\sigma_2)\subset\Omega$, from the first identity in the proof of \cite[Proposition~7]{Ca} one has
\begin{equation}\label{Ca7}
((Q_1-Q_2)Z_1,Y_2)_\Omega
=(Y_1,PY_2)_\Omega-(PY_1,Y_2)_\Omega.
\end{equation}
Moreover, using the estimate of the right hand of \eqref{Ca7} derived in \cite{Ca} (see the inequality before \cite[Proposition~7]{Ca}) and the support assumption of $\sigma_1-\sigma_2$, we can obtain that \begin{align}\label{Ca6}
&|(Y_1,PY_2)_\Omega-(PY_1,Y_2)_\Omega|\notag\\
&\leq  C( \|\nu\times E_1-\nu\times E_2\|_{TH(\partial\Omega)}+ \|\nu\times H_1-\nu\times H_2\|_{TH(\partial\Omega)})\|Y_2\|_{(1)}(\Omega) \notag\\
&\leq  C\dist({\mathcal C}_{1},{\mathcal C}_{2})\|\nu\times E_1 \|_{TH(\partial\Omega)}\|Y_2\|_{(1)}(\Omega)\notag\\
&\leq C\dist({\mathcal C}_{1},{\mathcal C}_{2})\|Y_1\|_{(1)}(\Omega)\|Y_2\|_{(1)}(\Omega),
\end{align}
where $(\nu\times E_2, \nu\times H_2)\in \mathcal{C}_2$.
The required estimate follows from (\ref{Ca7}) and (\ref{Ca6}).
\end{proof}

Let $\xi\in \mathbb{R}^3$. We select $e(j)\in \mathbb{R}^3$ satisfying $|e(1)|=|e(2)|=1$ and $e(j)\cdot\xi=e(1)\cdot e(2)=0$ for $j=1,2$. Let $R>0$. We choose
 $$
 \zeta(1) =-\frac{1}{2}\xi+
 i \left(\frac{R^2}{2}\right)^{1/2} e(1)+
 \left( \omega^2+\frac{R^2}{2}-
 \frac{|\xi|^2}{4}\right)^{1/2}e(2),
 $$
 \begin{equation*}
 \zeta(2) = \frac{1}{2}\xi-i\left(\frac{R^2}{2}\right)^{1/2} e(1)+\left( \omega^2+\frac{R^2}{2}-
 \frac{|\xi|^2}{4}\right)^{1/2}e(2), 
 \end{equation*}
where we assume 
$$
   \omega^2+\frac{R^2}{2}\geq \frac{|\xi|^2}{4}.
$$
Then 
$$
    \zeta(1)-\overline{\zeta}(2)=-\xi,\ \zeta(j)\cdot\zeta(j)=\omega^2,\ |\zeta(j)|^2=R^2+\omega^2,\ \ \mbox{for $j=1,2$}.
$$

To construct CGO solutions of Proposition \ref{cgo:max} with $Q=Q_1$ corresponding to $\sigma_1$, we choose $a(1)=0$ and
$$
b(1)=-i\frac{e(1)}{\sqrt{2}}+\frac{e(2)}{\sqrt{2}}.
$$
By direct calculation, we can see that
\begin{equation*}\label{0813}
     \|\omega^2+q_j\|_{(2s)}(B)+ \|\omega^2 I_8+Q_1\|_{(2s)}(B)\leq C_2\|\sigma_1\|_{(2s+2)}(B)\omega+C_3,
\end{equation*}
where $C_2$, $C_3$ are independent of $\omega$.
If we take
\begin{align}\label{a1}
  R+\omega> 2^{1/2}(C_0 C_2 \|\sigma_1\|_{(2s+2)}(B)\omega+C_0C_3),
\end{align}
then
$$
    |\zeta(1)|\geq 2^{-1/2}(R+\omega)>C_0\left( \|\omega^2+q_j\|_{(2s)}(B)+
    \|\omega^2 I_8+Q_1\|_{(2s)}(B)\right).
$$
By Proposition \ref{cgo:max}, there exist solutions
$$
Z_1=e^{i\zeta(1)\cdot x}(A(1)+\Psi_1)
$$
of $(-\Delta I_8+Q_1)Z_1=0$ with
$$
    A(1)=\frac{1}{|\zeta(1)|}
\left(
    \begin{tabular}{c}
     $0$\\
     $\omega b(1)$\\
   \hline
      $\zeta(1)\cdot b(1)$\\
      $0$
    \end{tabular}
\right)
$$
and
\begin{equation*}\label{R1}
            \|\Psi_1\|_{(2s)} (\Omega)  \leq
            \frac{C_0}{|\zeta(1)|} |A(1)|\|\omega^2 I_8+Q(1)\|_{(2s)}(B ).
\end{equation*}

Likewise, if
\begin{align}\label{a2}
   R+\omega> 2^{1/2}(C_0C_2  \|\sigma_2\|_{(2s+2)}(B)\omega+C_0C_3),
\end{align}
then one has
$$
    |\zeta(2)|>C_0 \|(\omega^2 I_8+Q(2)_2\|_{(2s)}(B),
$$
where $Q(2)_2$ is the matrix-valued function $Q(2)$ corresponding to $\sigma=\sigma_2$. By choosing $a^*=0$ and
$$
b^*=b(2)=i\frac{e(1)}{\sqrt{2}}+\frac{e(2)}{\sqrt{2}}
$$
in Proposition \ref{cgo:res} with $Q(2)$ corresponding to $\sigma_2$, there exist solutions
$$
     Y_2=e^{i\zeta(2)\cdot x}(A^*(2)+\Psi^*_2)
$$
of $(P+W_2^*)Y_2=0$ with
$$
    A^*(2) =\frac{1}{|\zeta(2)|}
\left(
    \begin{tabular}{c}
     $0$\\
     $0$\\
   \hline
      $\zeta(2)\cdot b(2)$\\
      $\zeta(2)\times b(2)$
    \end{tabular}
\right)
$$
and $\Psi^*_2=P\Psi_{2*}+iP(\zeta(2)) \Psi_{2*}-\overline{W_2}A_{2*}-
\overline{W_2}\Psi_{2*}$, where $A_{2*}=|\zeta(2)|^{-1}(0,\ \ b(2), \ \ 0, \ \ 0)^t$ and $\Psi_{2*}$ satisfies (\ref{4:R1}), (\ref{4:R2}). Moreover,
\begin{align}\label{S2}
      \|\Psi^*_2\|_{(2s)}(\Omega)&
\leq C_0\left(\frac{1}{|\zeta(2)|} + \frac{ \|\sigma_2 \|_{(2s+2)}(B)\omega }{|\zeta(2)|^2}\right) \|\omega^2 I_8+Q(2)_2\|_{(2s)}(B)\notag\\
&\quad+C_0\left(\frac{1 }{|\zeta(2)|} +\frac{ 
\|\sigma_2 \|_{(2s+2)}(B)\omega}{|\zeta(2)|^2}\right) \|W_2\|_{(2s)}(B)
\end{align}
(see \eqref{0809}).

For $Y_2$, we can obtain the following estimate.
\begin{lemma}
Let $Y_2$ be the CGO solution of $(P+W_2^*)Y_2=0$ as above. Then for $\omega\geq 1$ there exists a constant $C$ depending on $s,\Omega$ such that
\begin{align}\label{5:Y2}
&\|Y_2\|_{(1)}(\Omega)\leq
C\left(1+|\zeta(2)| \right) e^{  2^{-1/2}R}.
\end{align}
\end{lemma}

\begin{proof}

 Since $\nabla (e^{i\zeta(2)\cdot x}A^*(2))=i\zeta(2)e^{i\zeta(2)\cdot x}A^*(2)$ and
$|A^*(2)|\leq 2^{1/2},$
one gets
$$
   \|e^{i\zeta(2)\cdot x}A^*(2)\|_{(1)}(\Omega)\leq
 C( e^{2^{-1/2}R}|A^*(2)|+|\zeta(2)|e^{  2^{-1/2}R}|A^*(2)|)
\leq C(1+|\zeta(2)|)e^{ 2^{-1/2}R}.
$$
By direct calculation,
$$
\|\omega^2 I_8+Q(2)_2\|_{(2s)}(B )\leq C \omega,
\ \
\|W_2\|_{(2s)}(B)\leq C\omega ,
$$
where we used the assumption $\|\sigma_2\|_{(2s+2)}(\Omega)\le m<1$. Thus, from (\ref{S2}) we have
$
   \|\Psi^*_2\|_{(1)}(\Omega)
\leq  C\frac{\omega}{|\zeta(2)|}
$
which implies that
\begin{align*}
    \|e^{i\zeta(2)\cdot x}\Psi^*_2\|_{(1)}(\Omega)
&\leq \|e^{i\zeta(2)\cdot x}\Psi^*_2\|_{(0)}(\Omega)+
\|i\zeta(2)e^{i\zeta(2)\cdot x}\Psi^*_2\|_{(0)}(\Omega)+
\|e^{i\zeta(2)\cdot x}\nabla \Psi^*_2\|_{(0)}(\Omega)\\
&\leq C(1+|\zeta(2)|)\frac{\omega}{|\zeta(2)|}e^{  2^{-1/2}R} \leq C(1+|\zeta(2)|)e^{  2^{-1/2}R}.
\end{align*}
The proof is complete.
\end{proof}

Similarly, we can prove that
\begin{lemma}
Let $Z_1$ be the CGO solution of $(-\Delta I_8+Q_1)Z_1=0$ and $Y_1=(P-W_1^t)Z_1$ as above. Then for $\omega\geq 1$ there exists a constant $C$ depending on $s, \Omega$ such that
\begin{align}\label{5:Y1}
&\|Y_1\|_{(1)}(\Omega)\leq
C\left(1+|\zeta(1)| \right)^2 e^{  2^{-1/2}R}.
\end{align}
\end{lemma}
\begin{proof}
Substituting $ Z_1=e^{i\zeta(1)\cdot x}(A(1)+\Psi_1)$ gives
$$
Y_1=(P-W_1^t)Z_1=(P-W_1^t)(e^{i\zeta(1)\cdot x}(A(1)+\Psi_1)).
$$
We evaluate 
\begin{align*}
&\|(P-W_1^t)(e^{i\zeta(1)\cdot x}A(1))\|_{(1)}(\Omega)\\
&\leq C |\zeta(1)|  \| e^{i\zeta(1)\cdot x}A(1)   \|_{(1)}(\Omega)+C \| W_1^te^{i\zeta(1)\cdot x}A(1)   \|_{(1)}(\Omega)\\
&\leq C(1+|\zeta(1)|  )^2e^{2^{-1/2} R}
\end{align*}
and
\begin{align*}
&\|(P-W_1^t)(e^{i\zeta(1)\cdot x}\Psi_1)\|_{(1)}(\Omega)\\
&\leq  C |\zeta(1)|  \| e^{i\zeta(1)\cdot x}\Psi_1   \|_{(1)}(\Omega)+C \| W_1^t e^{i\zeta(1)\cdot x}\Psi_1 \|_{(1)}(\Omega)    + C \|e^{i\zeta(1)\cdot x}P\Psi_1 \|_{(1)}(\Omega)\\
&\leq C(1+|\zeta(1)| )^2e^{2^{-1/2} R},
\end{align*}
which completes the proof.
\end{proof}

We now choose $m<1$ satisfying $2^{1/2}C_0 C_2 m<1$ and $ R>2^{1/2}C_0 C_3$, i.e., \eqref{a1} and \eqref{a2} hold.  Combining \eqref{id}, \eqref{5:Y2} and \eqref{5:Y1} yields that
\begin{align}
\label{id2}
    |((Q_1-Q_2)Z_1,Y_2)_\Omega|
&\leq C\dist({\mathcal C}_{1},{\mathcal C}_{2})\|Y_1\|_{(1)}(\Omega) \|Y_2\|_{(1)}(\Omega)\notag\\
&\leq C  (1+|\zeta(2)|)^3 e^{2^{ 1/2}R}\dist({\mathcal C}_{1},{\mathcal C}_{2}).
\end{align}

\begin{lemma}
\label{5:gamma}
Let $s>3/2$ and $\xi=r e$ with $r\geq 0, |e|=1$. Under the assumptions of Theorem \ref{caroq}, there exist a constant $C_\Omega$ depending on $\Omega$ only and a constant $C$ depending on $s,\Omega$ such that
\begin{align}\label{fourmax}
 | (\hat{\sigma}_1-\hat{\sigma}_2)(re)|
&\leq  \frac{ C_\Omega +Cm}{\beta(\xi)^2 } {\mathcal G}({\nabla^2\tilde\sigma})(\xi)
      +  \frac{C_\Omega \beta(\xi) +Cm(\omega+\beta(\xi))}{\beta(\xi)^2   } {\mathcal G}(\nabla\tilde\sigma)(\xi)   \notag\\
      &\quad+   \frac{C_\Omega \omega^{-1}\beta(\xi) +Cm(\omega^2+\omega\beta(\xi))}{\beta(\xi)^2   }{\mathcal G}(\tilde\sigma)(\xi)  
      + C{ \frac{(\omega^2+ R^2)^{7/2}}{\omega\beta(\xi)^2}   } e^{{ 2^{ 1/2}}R}\dist({\mathcal C}_{1},{\mathcal C}_{2}),
\end{align}
where we denote
$$
\beta(\xi)=R+\left(2\omega^2+ R^2 -\frac{|\xi|^2}{2} \right)^{1/2}, 
$$
$\tilde\sigma=\sigma_1-\sigma_2$, and 
\begin{align*}
{\mathcal G}(f)(\xi)^2=\int \langle x\rangle^{-4s}  \int \langle -\eta+\xi-x\rangle ^{-2 s}  |\hat{f}(\eta)|^2 d\eta dx. 
\end{align*}
\end{lemma}
\begin{proof}
Using $Z_1=e^{i\zeta(1)\cdot x}(A(1)+\Psi_1)$, $Y_2=e^{i\zeta(2)\cdot x}(A^*(2)+\Psi^*_2)$
and the definition of $Q_1, Q_2$, we have
\begin{align}\label{four}
      & ((Q_1-Q_2)Z_1,Y_2)_\Omega= \int_\Omega e^{-i\xi\cdot x} (\overline{A^*(2)}+\overline{\Psi^*_2})^t
      (Q_1-Q_2)(A(1)+\Psi_1)dx\notag\\
&= \int_\Omega e^{-i\xi\cdot x}  \overline{A^*(2)} ^t (Q_1-Q_2) A(1) dx  +  \int_\Omega e^{-i\xi\cdot x}  \overline{\Psi^*_2} ^t (Q_1-Q_2) A(1) dx \notag\\
&\quad +  \int_\Omega e^{-i\xi\cdot x}  (\overline{A^*(2)}+\overline{\Psi^*_2})^t (Q_1-Q_2) \Psi_1 dx      \notag\\
&=:I_1+I_2+I_3.
\end{align}
We will evaluate each term in \eqref{four}. We begin with $I_1$. Since $A(1)=|\zeta(1)|^{-1} (0,\ \ \omega b(1),\ \ \zeta(1)\cdot b(1),\ \ 0)^t$ and $A^*(2)=|\zeta(2)|^{-1}(0,\ \ 0, \ \ \zeta(2)\cdot b(2),\ \ \zeta(2)\times b(2))^t$, we deduce that
$$
    I_1
    = \frac{(\zeta(1)\cdot b(1))(\overline{\zeta(2)} \cdot \overline{b (2)})}{|\zeta(1)||\zeta(2)|}
    \left(\int_\Omega  e^{-i\xi\cdot x} (\kappa_2^2-\kappa_1^2)dx\right)
+
$$
\begin{equation}
\label{4:I1}
 2\omega\frac{\overline{\zeta(2)} \cdot \overline{b(2)} }{|\zeta(1)||\zeta(2)|} \left(\int_\Omega  e^{-i\xi\cdot x} D(\kappa_2-\kappa_1)\cdot b(1) dx\right).
\end{equation}

To bound $I_2$, we recall that $\Psi^*_2=P\Psi_{2*}+iP(\zeta(2)) \Psi_{2*}-\overline{W_2}A_{2*}-\overline{W_2}\Psi_{2*}$, then
\begin{align*}
I_2&=\int_\Omega e^{-i\xi\cdot x}  (\overline{P\Psi_{2*}})^t (Q_1-Q_2) A(1) dx
+\int_\Omega e^{-i\xi\cdot x}  (\overline{iP(\zeta(2))\Psi_{2*}})^t (Q_1-Q_2) A(1) dx\notag\\
&\quad+\int_\Omega e^{-i\xi\cdot x}  (\overline{-\overline{W_2}A_{2*}})^t (Q_1-Q_2) A(1) dx
+\int_\Omega e^{-i\xi\cdot x}  (\overline{-\overline{W_2}\Psi_{2*}})^t (Q_1-Q_2) A(1) dx\notag\\
&=:J_1+J_2+J_3+J_4.
\end{align*}
We now bound $J_k\ ,k=1,2,3,4$. Applying (\ref{4:R1}), (\ref{4:R2}) and using that $m<1,  \omega\leq |\zeta(j)|$, we can deduce that
\begin{equation}
\label{4:PRR}
\|\Psi_{2*}\|_{(2s)}(\Omega)
\leq C\frac{m\omega}{|\zeta(2)|^2}\ \ ,\ \
\|\overline{P\Psi_{2*}}\|_{(2s)}(\Omega)
\leq C\frac{m\omega}{|\zeta(2)|}.
\end{equation}
By an argument similar to Lemma \ref{lemma3.4}, we have
\begin{align*}
|J_1|
&\leq   C \mathcal{Q}(\xi) \|\overline{P\Psi_{2*}}\|_{(2s)}(\Omega)\\
 &\leq C\frac{m\omega }{|\zeta(2)|} \mathcal{Q}(\xi) ,
\end{align*}
where 
$$
     \mathcal{Q}(\xi)=\left(\frac{1}{|\zeta(1)|}{\mathcal G}(\nabla^2\tilde\sigma)(\xi)+\frac{\omega+\zeta(1)\cdot b(1)}{|\zeta(1)|}{\mathcal G}(\nabla\tilde\sigma)(\xi) 
     +\frac{\omega^2+\omega(\zeta(1)\cdot b(1))}{|\zeta(1)|}{\mathcal G}(\tilde\sigma)(\xi) \right).
$$

Using the bound
$ \|\overline{P(\zeta(2))\Psi_{2*}}\|_{(2s)}(\Omega)
\leq C m\omega |\zeta(2)|^{-1}$ and  following a similar argument in Lemma \ref{lemma3.4}, we can derive
$$
|J_2|
\leq  C\frac{m\omega }{|\zeta(2)|} \mathcal{Q}(\xi).
$$
For $J_4$, it follows from the definition of $W_2$ that
$ \|W_2\|_{(2s)}(\Omega) \leq C \omega,$
then applying (\ref{4:PRR}) and the proof of Lemma \ref{lemma3.4}, we yield
$$
|J_4|
\leq  C\frac{m\omega }{|\zeta(2)|} \mathcal{Q}(\xi).
$$
Finally, we would like to estimate $J_3$. Note that
\begin{equation}\label{08120}
W_2\overline{A_{2*}}
=\frac{1}{|\zeta(2)|}\left(\kappa_2 I_8+\frac{1}{2}
\left(
\begin{tabular}{cc|cc}
  0 & 0 & 0 & $D\alpha_2\cdot$ \\
  0& 0&   $D\alpha_2 $& $D\times \alpha_2$\\
  \hline
  0 & 0 &0 &0 \\
    0 & 0 &0 &0 \\
\end{tabular}
       \right) \right)
\left(
  \begin{array}{c}
    0 \\
    \overline{b(2)} \\
    \hline
    0  \\
    0 \\
  \end{array}
\right)
=\frac{\omega \gamma_2^{1/2}}{|\zeta(2)|}
\left(
  \begin{array}{c}
    0 \\
    \overline{b(2)} \\
   \hline
    0  \\
    0 \\
  \end{array}
\right).
\end{equation}
In view of the definition of $Q_j$, direct calculations show that
\begin{align}\label{08121}
&(Q_1-Q_2)A(1)\notag\\
&=\frac{\omega}{2|\zeta(1)|}
\left(
  \begin{array}{c}
    0 \\
   2(\nabla^2(\alpha_1-\alpha_2))b(1)-
   (\Delta(\alpha_1-\alpha_2))b(1)\\
   \hline
    0  \\
    0 \\
  \end{array}
\right)
-\frac{i\omega(\sigma_1-\sigma_2)}{|\zeta(1)|}
\left(
  \begin{array}{c}
    0 \\
   \omega b(1) \\
   \hline
    \zeta(1)\cdot b(1)  \\
    0 \\
  \end{array}
\right)\notag\\
&\quad-
\frac{1}{|\zeta(1)|}
\left(
  \begin{array}{c}
    0  \\
    \frac{1}{4}\omega(D\alpha_1\cdot D\alpha_1- D\alpha_2\cdot D\alpha_2) b(1)+2(\zeta(1)\cdot b(1))D(\kappa_1-\kappa_2) \\
    \hline
    2\omega b(1) \cdot D(\kappa_1-\kappa_2) \\
    0 \\
  \end{array}
\right).
\end{align}
Now putting \eqref{08120}, \eqref{08121} together and using the fact that $\overline{b(2)}\cdot b(1)=0,$ we have
\begin{align*}
 &(W_2\overline{A_{2*}})^t(Q_1-Q_2)A(1)\notag\\
 &=\frac{\omega^2\gamma_2^{1/2}\overline{b(2)}^t
 (\nabla^2(\alpha_1-\alpha_2))b(1)  -
   2 \omega \gamma_2^{1/2}( \zeta(1)\cdot b(1) )\overline{b(2)}^t D(\kappa_1-\kappa_2) }{|\zeta(1)|
   |\zeta(2)|},
\end{align*}
which implies, combined with $\omega\geq 1$, $m<1$, $|\zeta(j)|\geq \omega$, that
$$
|J_3|\leq C_\Omega\mathcal{Q}_0(\xi) 
$$
where
the constant $C_\Omega$ only depends on $\Omega$ and
$$
    \mathcal{Q}_0(\xi)= \frac{1}{|\zeta(1)||\zeta(2)|} (\omega{\mathcal G}(\nabla^2\tilde\sigma)(\xi)+\omega(\zeta(1)\cdot b(1)){\mathcal G}(\nabla\tilde\sigma)(\xi)+(\zeta(1)\cdot b(1) ){\mathcal G}(\tilde\sigma)(\xi)).
$$
Combining estimates for $J_1$ to $J_4$ yields
\begin{align}
\label{5:w2}
     |I_2|&\leq C\frac{m\omega }{|\zeta(2)|} \mathcal{Q}(\xi)   +   C_\Omega\mathcal{Q}_0(\xi).
\end{align}

To evaluate $I_3$, from (\ref{S2}), 
$$
\|\Psi_1\|_{(2s)}(\Omega) \leq  \frac{C}{|\zeta(1)|} |A(1)| \|\omega^2 I_8+Q_1\|_{(2s)}(B) \leq
C\frac{\omega b(1)+\zeta(1)\cdot b(1)}{|\zeta(1)|}\frac{ m \omega }{|\zeta(1)|} ,
$$
it follows that
\begin{equation}
\label{5:w3}
     |I_3|\leq C\frac{m\omega }{|\zeta(2)|} \mathcal{Q}(\xi).
\end{equation}
Note that $\zeta(j)\cdot b(j)>0$ for $j=1,2$. Combining (\ref{id2}), (\ref{four}), (\ref{4:I1}), (\ref{5:w2}) and (\ref{5:w3}), we get
\begin{align}\label{5:10}
     &\frac{(\zeta(1)\cdot b(1))(\zeta(2)\cdot b(2))}
     {|\zeta(1)||\zeta(2)|}\left|\left(\int_\Omega
     e^{-i\xi\cdot x} (\kappa_2^2-\kappa_1^2)dx\right)\right|\notag\\
&\leq  2\omega\frac{\zeta(2)\cdot b(2)}{|\zeta(1)||\zeta(2)|}\left|\int_\Omega  e^{-i\xi\cdot x} D(\kappa_2-\kappa_1)\cdot b(1) dx\right| 
  + C\frac{m\omega }{|\zeta(2)|} \mathcal{Q}(\xi)+C_\Omega \mathcal{Q}_0(\xi)  +F,
\end{align}
where
$
F=C (1+|\zeta(2)|)^3 e^{ 2^{-1/2}R}
\dist({\mathcal C}_{1},{\mathcal C}_{2}).
$
Substituting $\kappa_j=\omega\gamma_j^{1/2}$ into the estimate (\ref{5:10}), we obtain that
\begin{align*}
&\omega\frac{(\zeta(1)\cdot b(1))(\zeta(2)\cdot b(2))}{|\zeta(1)||\zeta(2)|}\left|\int_\Omega  e^{-i\xi\cdot x} (\sigma_2-\sigma_1)dx\right|\\
&\leq 2\omega^2 \left|\int_\Omega e^{-i\xi\cdot x} D(\gamma_2^{1/2}-\gamma_1^{1/2})\cdot b(1)dx\right|
\frac{\zeta(2)\cdot b(2)}{|\zeta(1)||\zeta(2)|} + C\frac{m\omega }{|\zeta(2)|} \mathcal{Q}(\xi)+C_\Omega\mathcal{Q}_0(\xi)  +F,
\end{align*}
which leads to
\begin{align}\label{ss1}
| (\hat{\sigma}_1-\hat{\sigma}_2)(\xi)|
&\leq  C_\Omega \frac{1}{(\zeta(1)\cdot b(1)) 	}({\mathcal G}(\nabla\tilde\sigma)(\xi) +\omega^{-1}{\mathcal G}(\tilde\sigma)(\xi))  +    C\frac{m|\zeta(1)| }{(\zeta(1)\cdot b(1))(\zeta(2)\cdot b(2))} \mathcal{Q}(\xi)\notag\\
&\quad+ C_\Omega \frac{ |\zeta(1)||\zeta(2)|}{\omega(\zeta(1)\cdot b(1))(\zeta(2)\cdot b(2))} \mathcal{Q}_0(\xi)
  + \omega^{-1}\frac{|\zeta(1)||\zeta(2)|}{(\zeta(1)\cdot b(1))(\zeta(2)\cdot b(2))}F\notag\\
&\leq g_1(\xi)
 + g_2(\xi)  + \omega^{-1}\frac{|\zeta(1)||\zeta(2)|}{(\zeta(1)\cdot b(1))(\zeta(2)\cdot b(2))}F\notag\\
\end{align}
with
  \begin{align*}
      g_1(\xi)&=C_\Omega  \frac{1}{ (\zeta(1)\cdot b(1))(\zeta(2)\cdot b(2)) }{\mathcal G}(\nabla^2\tilde\sigma)(\xi)
     +  C_\Omega\frac{1}{ (\zeta(1)\cdot b(1))  }{\mathcal G}(\nabla\tilde\sigma)(\xi)+ C_\Omega\frac{1}{ \omega(\zeta(1)\cdot b(1))  }{\mathcal G}(\tilde\sigma)(\xi);\\
     g_2(\xi)&=
     C \frac{m}{ (\zeta(1)\cdot b(1))(\zeta(2)\cdot b(2)) } {\mathcal G}(\nabla^2\tilde\sigma)(\xi) +
      C  \frac{m(\omega + (\zeta(1)\cdot b(1)) )  }{ (\zeta(1)\cdot b(1))(\zeta(2)\cdot b(2))  }{\mathcal G}(\nabla\tilde\sigma)(\xi)\\
      &\quad+  C  \frac{m(\omega^2+\omega(\zeta(1)\cdot b(1)) )  }{ (\zeta(1)\cdot b(1))(\zeta(2)\cdot b(2))  }{\mathcal G}(\tilde\sigma)(\xi).
  \end{align*}
 
Since 
$$
2^{1/2}( \zeta(1)\cdot b(1))=2^{-1/2}R+\left(\omega^2+\frac{R^2}{2}-\frac{|\xi|^2}{4} \right)^{1/2},\ \ \ (1+|\zeta(2)|) \leq 2(\omega^2+R^2)^{1/2},
$$
we have  
\begin{align}\label{g12}
      g_1(\xi)&\leq  C_\Omega  \frac{1}{\beta(\xi)^2 } {\mathcal G}(\nabla^2\tilde\sigma)(\xi)
      +  C_\Omega\frac{1}{\beta(\xi)  }{\mathcal G}(\nabla\tilde\sigma)(\xi) +  C_\Omega\frac{1}{ \omega\beta(\xi)  }{\mathcal G}(\tilde\sigma)(\xi);\\
     g_2(\xi)&\leq 
    C   \frac{m}{\beta(\xi)^2 } {\mathcal G}(\nabla^2\tilde\sigma)(\xi) +
      C \frac{m(\omega+\beta(\xi))}{\beta(\xi)^2   }{\mathcal G}(\nabla\tilde\sigma)(\xi) +  C \frac{ m(\omega^2+\omega \beta(\xi)) }{\beta(\xi)^2 }{\mathcal G}(\tilde\sigma)(\xi)
  \end{align}
  and 
  \begin{equation}\label{ff12}
     F\leq C ( \omega^2+R^2)^{3/2} e^{  2^{1/2}R}  \dist({\mathcal C}_{1},{\mathcal C}_{2}),
  \end{equation}
where
  $$
  \beta(\xi)=R+\left(2\omega^2+ R^2 -\frac{|\xi|^2}{2} \right)^{1/2} .
  $$
The estimates (\ref{ss1})- (\ref{ff12}) complete the proof.
\end{proof}
 
Recall that $\epsilon= \dist({\mathcal C}_{1}, {\mathcal C}_{2}),\ \mathcal{E}=-\log\epsilon$. Before proving our main result, we need the following lemma.
\begin{lemma}
Let $R^*>2^{1/2}C_0C_3$ with $C_0$ and $C_3$ defined in Proposition~\ref{cgo:res} and \eqref{0813}, respectively. There exist a constant $C_\Omega$ depending only on $\Omega$ and a constant $C$ depending on $s, \Omega$ such that the following estimates hold: if $0\leq r\leq \omega+R^*$, then
\begin{align}\label{mm1}
| (\hat{\sigma}_1-\hat{\sigma}_2)(r e)|
   &\leq   \frac{ C_\Omega+Cm }{ (R^*+\omega)^2} {\mathcal G}(\nabla^2\tilde\sigma)(\xi)+   \frac{ C_\Omega+Cm  }{  R^*+\omega }{\mathcal G}(\nabla\tilde\sigma)(\xi)\notag\\
   &\quad +\left(   \frac{ C_\Omega  }{\omega(R^*+\omega)}  +Cm\right){\mathcal G}(\tilde\sigma)(\xi)+
   C\omega^{-1}( \omega^2+R^{2*})^{3/2}  e^{ 2^{ 1/2} R^* }   \epsilon ;
\end{align}
if $r\geq \omega+R^*$, then
\begin{align}\label{mm2}
| (\hat{\sigma}_1-\hat{\sigma}_2)(r e)|
   &\leq     \frac{  C_\Omega+Cm  }{ (r+\omega)^2} {\mathcal G}(\nabla^2\tilde\sigma)(\xi)+   \frac{ C_\Omega+Cm  }{ r+\omega } {\mathcal G}(\nabla\tilde\sigma)(\xi)\notag\\
      &\quad +\left(   \frac{ C_\Omega  }{\omega(r+\omega)}  +Cm\right){\mathcal G}(\tilde\sigma)(\xi)+
   C\omega^{-1}(\omega^2+r^2  )^{3/2}   e^{ 2^{ 1/2}r}   \epsilon.
\end{align}
\end{lemma}

\begin{proof}
We choose $R=R^*$ when $0\leq r\leq \omega+R^*$ and  $R=r$ when $r\geq \omega+R^*$ in the estimate (\ref{fourmax}).
\end{proof}

Now we are ready to show Theorem~\ref{main}. \\
\medskip\noindent
\emph{Proof of Theorem \ref{main}.} 
Let $\tilde \sigma=\sigma_1-\sigma_2$. Using polar coordinates, we obtain
\begin{align}\label{F0b}
    \|\tilde \sigma\|^2_{(-s )}(\Omega) 
     &\leq \int_{|\xi|\leq \omega+R^*}   |\hat{\tilde \sigma}(\xi)|^2(1+|\xi|^2)^{-s }   d \xi \notag\\
     &\quad+\int_{\omega+R^*\leq|\xi|\leq T}  |\hat{\tilde \sigma}(\xi)|^2(1+|\xi|^2)^{-s }   d\xi\notag\\
     &\quad+\int_T^\infty\int_{|e|=1}  |\hat{\tilde \sigma}(r e)|^2(1+r^2)^{-s } r^2 d e dr\notag\\
     &=:I_1+I_2+I_3.
\end{align}

We begin with bounding $I_3$. Since $\supp \tilde \sigma \subset \Omega$ by the H\"older's inequality
$|\hat{\tilde \sigma}(\xi)|\leq C\|\tilde \sigma\|_{(0)}(\Omega)\leq m\leq 1$ and so
\begin{align}\label{F1b}
I_3& \leq
C \|\tilde \sigma\|^2_{(0)}(\Omega)\int_T^\infty\int_{|e|=1}   (1+r^2)^{-s }  r^2 de dr\notag\\
     &\leq C \|\tilde \sigma\|^2_{(0)}(\Omega)T^{-(2s-3)}
     \leq C  T^{-(2s-3)}.
\end{align}

Before evaluating $I_1$ and $I_2$ terms, we need the following estimate.
Let $A=\{\eta:\ |-\eta+\xi-x|\leq |\eta|/2\}.$ By direct computation, we have, for $0\le\alpha\leq 2$,
\begin{align}\label{k1}
&  \int \langle \xi\rangle^{ -2s }   \int \langle x\rangle^{-4s}  \left(\int_{A^c} \langle -\eta+\xi-x\rangle ^{-2 s} |\widehat{  \nabla^\alpha \tilde{\sigma}}(\eta)| ^2 d\eta \right)        dx  d\xi \notag\\
& \leq  C \int \langle \xi\rangle^{ -2s }   \int \langle x\rangle^{-4s}  \left(\int_{A^c} \langle \eta \rangle ^{-2 s} |\widehat{\nabla^\alpha \tilde{\sigma}}(\eta)|^2 d\eta \right)        dx  d\xi\notag\\
&\leq C\| \nabla^\alpha\tilde{\sigma}\|^2_{(-s)} .
\end{align}

On the other hand, by using the fact that $A\subset \{\eta;\  \frac{2}{3}|\xi-x|\leq |\eta|\leq 2|\xi-x|\}$ and $\langle x\rangle^{-2s}\langle \xi\rangle^{-2s} \leq 2^s \langle x-\xi\rangle^{-2s}  $, one has
\begin{align}\label{k2}
&  \int \langle \xi\rangle^{ -2s}   \int \langle x\rangle^{-4s}  \left(\int_{A } \langle -\eta+\xi-x\rangle ^{-2 s} |\widehat{  \nabla^\alpha \tilde{\sigma}}(\eta)| ^2 d\eta \right)        dx  d\xi  \notag\\
&\leq   \int \langle \xi\rangle^{ -2s}   \int \langle x\rangle^{-4s}   \int_{\frac{2}{3}|\xi-x|\leq |\eta|\leq 2|\xi-x|} \langle -\eta+\xi-x\rangle ^{-2 s} |\widehat{\nabla^\alpha \tilde{\sigma}}(\eta)| ^2 d\eta        dx  d\xi \notag \\
&\leq C \int   \langle x\rangle^{- 2s}   \int \int_{\frac{2}{3}|\xi-x|\leq |\eta|\leq 2|\xi-x|} \langle x-\xi\rangle^{- 2s} \langle -\eta+\xi-x\rangle ^{-2 s} |\widehat{  \nabla^\alpha \tilde{\sigma}}(\eta)| ^2 d\eta        d\xi dx    \notag \\
& \leq C  \int \langle x\rangle^{- 2s}   \int \int_{\frac{2}{3}|\xi-x|\leq |\eta|\leq 2|\xi-x|} \langle \eta \rangle^{- 2s} \langle -\eta+\xi-x\rangle ^{-2 s} |\widehat{  \nabla^\alpha \tilde{\sigma}}(\eta)| ^2 d\eta        d\xi dx  \notag  \\
& \leq C  \int \left(   \int \int  \langle x\rangle^{- 2s}  \langle -\eta+\xi-x\rangle ^{-2 s} d\xi dx \right) |\widehat{  \nabla^\alpha \tilde{\sigma}}(\eta)| ^2 \langle \eta \rangle^{- 2s} d\eta \notag\\
 &\leq C\| \nabla^\alpha \tilde{\sigma}\|^2_{(-s)}.
\end{align}
  We obtain from (\ref{k1}) and (\ref{k2}) that 
\begin{align}\label{newQb}
      \int  \langle \xi\rangle^{ -2s}|{\mathcal G}(\nabla^{\alpha}\tilde\sigma)(\xi)|^2 d\xi \leq C\| \nabla^\alpha \tilde{\sigma}\|^2_{(-s)}.\end{align}

Next, by using (\ref{mm1}), (\ref{k1}), (\ref{k2}) and (\ref{newQb}), we yield
\begin{align}\label{I1}
I_1&\leq \int_{|\xi|\leq \omega+R^*}  \Big( \frac{(C_\Omega+Cm)^2}{ (R^*+\omega)^4}|{\mathcal G}(\nabla^2\tilde\sigma)(\xi)|^2+   \frac{(C_\Omega+Cm)^2 }{  (R^*+\omega)^2 }|{\mathcal G}(\nabla\tilde\sigma)(\xi)|^2\notag\\
&\quad + \left(   \frac{ C_\Omega  }{\omega(R^*+\omega)}  +Cm\right)^2 |{\mathcal G}(\tilde\sigma)(\xi)|^2+C\omega^{-2}(\omega^2+R^{*2})^3 e^{2\sqrt{2}R^*}\epsilon ^2\Big) \langle \xi\rangle^{ -2s}    d\xi\notag\\
&\leq  Cm^2  \|\tilde \sigma\|^2_{(-s )} +  \frac{C}{(R^*+\omega)^2}  \|\tilde \sigma\|^2_{(-s+2)}  
 +C\omega^{-2}(\omega^2+R^{*2})^3 e^{2\sqrt{2}R^*}\epsilon ^2 \int_{|\xi|\leq \omega+R^*}\langle \xi\rangle^{ -2s}    d\xi \notag\\
&\leq   Cm^2  \|\tilde \sigma\|^2_{(-s )} +  \frac{Cm^2}{(R^*+\omega)^2}  
+C\omega^{-2}(\omega^2+R^{*2})^3 e^{2\sqrt{2}R^*} \epsilon^2.
\end{align}
Similarly, we deduce that 
\begin{align}\label{I2}
I_2&\leq \int_{\omega+R^*\leq |\xi|\leq T} \Big( \frac{(C_\Omega+Cm)^2}{ (r+\omega)^4}|{\mathcal G}(\nabla^2\tilde\sigma)(\xi)|^2+   \frac{ {(C_\Omega+Cm)^2}}{  (r+\omega)^2 }|{\mathcal G}(\nabla\tilde\sigma)(\xi)|^2 \notag\\
&\quad + \left(   \frac{ C_\Omega  }{\omega(r+\omega)}  +Cm\right)^2 |{\mathcal G}(\tilde\sigma)(\xi)|^2+\omega^{-2}(\omega^2+r^2)^3 e^{2\sqrt{2}r}\epsilon ^2\Big) \langle \xi\rangle^{ -2s}     d\xi\notag\\
&\leq  Cm^2  \|\tilde \sigma\|^2_{(-s )} +  \frac{C}{(R^*+\omega)^2}  \|\tilde \sigma\|^2_{(-s+2)}  +\omega^{-2}(\omega^2+T^2)^3 e^{2\sqrt{2}T} \epsilon^2  \int_{\omega+R^*\leq |\xi|\leq T}\langle \xi\rangle^{ -2s}   d\xi \notag\\
&\leq   Cm^2  \|\tilde \sigma\|^2_{(-s )} +  \frac{Cm^2}{(R^*+\omega)^2}    +\omega^{-2}(\omega^2+T^2)^3 e^{2\sqrt{2}T} \epsilon^2 .
\end{align}
Therefore, we conclude from (\ref{F0b}), (\ref{F1b}), (\ref{I1}) and (\ref{I2}) that 
\begin{lemma}
For any $T\ge\omega+R^*$, we have
 $$
       \|\tilde \sigma\|^2_{(-s )}(\Omega)
   \leq  Cm^2  \|\tilde \sigma\|^2_{(-s )} +  \frac{Cm^2}{(R^*+\omega)^2}  
   $$
        \begin{equation}\label{pre:qb}
   + C\omega^{-2}(\omega^2+T^2)^3 ( e^{2\sqrt{2}R^*}+ \chi(T) e^{2\sqrt{2}T} )\epsilon^2
       +C  T^{-(2s-3)},
 \end{equation}
where $\chi(T)\leq 1$ and $\chi(T)=0$ if $T=\omega+R^*$.
\end{lemma}

We then choose $m<1$ satisfying $Cm^2<1/2$. Thus, the first term on the right hand side of (\ref{pre:qb})
can be absorbed by the left hand side. We obtain
 \begin{align*} 
   \|\tilde \sigma\|^2_{(-s)}(\Omega)  \leq   \frac{C}{(R^*+\omega)^2}  
      + C\omega^{-2}(\omega^2+T^2)^3 ( e^{2\sqrt{2}R^*}+ \chi(T) e^{2\sqrt{2}T} )\epsilon^2
  +C  T^{-(2s-3)},
\end{align*}
Let $\mathcal{E}>4C_0C_3$, we take $ 2\sqrt{2}R^* =\mathcal{E}$ 
so that $R^*$ still satisfies $R^*> 2^{1/2}C_0C_3$ and we have
 \begin{align} \label{pre:q1}
   \|\tilde \sigma\|^2_{(-s)}(\Omega)  \leq   \frac{C}{(\mathcal{E}+\omega)^2}  
      + C\omega^{-2}(\omega^2+T^2)^3 ( \epsilon^{-1}+ \chi(T) e^{2\sqrt{2}T} )\epsilon^2
  +C  T^{-(2s-3)}.
\end{align}

We consider the following two cases:
\begin{align*}
    \text{(i)}\ \omega+R^*\leq \frac{1}{2}\mathcal{E}\ \ \hbox{and}\ \ \text{(ii)}\ \omega+\mathcal{R}^*\geq \frac{1}{2}\mathcal{E}.
\end{align*}
In the  case (i) we choose
$T=\frac{\mathcal{E}}{2}.$
By the same method as in the proof of Theorem~\ref{mainatt},
there exists $C_4$ such that
\begin{equation}\label{mres1}
\omega^{-2}(\omega^2+T^2)^3 e^{2\sqrt{2}T} \epsilon^2\leq C_4 (\omega+\mathcal{E})^{-(2s-3)}
\end{equation}
and
\begin{equation}\label{mres2}
 T^{-(2s-3)}\leq C_4 (\omega+\mathcal{E})^{-(2s-3)}.
\end{equation}
Therefore, from (\ref{pre:q1})-(\ref{mres2}), one has
\begin{align}\label{maxwe}
      \|\tilde \sigma\|^2_{(-s)}(\Omega)  \leq   \frac{C}{(\mathcal{E}+\omega)^2} +    C\omega^{-2}(\omega^2+\mathcal{E}^2)^3\epsilon
      +C(\omega+\mathcal{E})^{-(2s-3)} .
\end{align}

Next we consider case (ii). We choose $T=\omega+R^*$ so that $\chi(T)=0$.
Then  we have
\begin{equation*}
            \|\tilde \sigma\|^2_{(-s)}(\Omega)  \leq   \frac{C}{(\mathcal{E}+\omega)^2} +    C\omega^{-2}(\omega^2+(\omega+R^*)^2)^3\epsilon
                +C(\omega+R^*)^{-(2s-3)} .
\end{equation*}
Since $R^*=2^{-3/2}\mathcal{E}$, we have 
 the same estimate \eqref{maxwe}.  
We complete the proof of Theorem \ref{main}.

\section{ Conclusion}

The next analytic task is to get rid of smallness assumptions of Theorems 2.1, 2.2 on attenuation and conductivity. It possibly will need a more careful reduction to a vectorial Schr\"odinger equation using some results in \cite{OS}. It looks realistic to extend our increasing stability results for the conductivity coefficient onto the case of partial data
\cite{Ca0}. To do so one can try to adjust the methods (in particular a new construction of CGO solutions) of the recent paper \cite{L}. 

It is very important to verify if our theoretical prediction will result in better numerical reconstruction of the important conductivity coefficient for higher frequencies. After successful
numerical testing one can expect a substantial improvements in the electrical impedance tomography, which so far was cheap, non destructive imaging technique with unfortunately pretty low resolution.


\begin{thebibliography}{99}
\bibitem{A1}
     \newblock G. Alessandrini, 
     \newblock \emph{Stable determination of conductivity by boundary measurements},
     \newblock Appl. Anal.,  \textbf{27} (1988), 153--172.

\bibitem{Ca0}
\newblock P. Caro,
\newblock \emph{On an inverse problem in electromagnetism with local data: stability and uniqueness}, \newblock Inverse Problems and Imaging, \textbf{5} (2011), 297--322.

\bibitem{Ca}
     \newblock P. Caro, 
     \newblock \emph{Stable determination of the electromagnetic coefficients by boundary measurements},
     \newblock Inverse Problems, \textbf{26} (2010), 105014.

\bibitem{COS}
     \newblock P. Caro, P. Ola and M. Salo, 
     \newblock \emph{Inverse boundary value problem for Maxwell equations with local data},
     \newblock Comm. PDE, \textbf{34} (2009), 1425--64.

\bibitem{FSU}
     \newblock J. Feldman, M. Salo and G. Uhlmann, 
     \newblock \emph{Calder\'on problem: An introduction to inverse problems},
     \newblock Preliminary notes on the book in preparation.

\bibitem{Ibook}
\newblock V. Isakov,
\newblock Inverse Problems for Partial Differential Equations,
\newblock Springer- Verlag, New York, 2006.

\bibitem{I}
     \newblock V. Isakov,
     \newblock \emph{Increasing stability for the Schr\"odinger potential from the Dirichlet-to-Nequmann map},
     \newblock DCDS-S, (2011), 631--640.

\bibitem{INUW}
     \newblock V. Isakov and S. Nagayasu, G. Uhlmann and J.-N. Wang,
     \newblock \emph{Increasing stability of the inverse boundary value problem for the Schr\"odinger equation},
     \newblock  Contemp. Math., \textbf{615} (2014), 131-143.

\bibitem{iw14}
     \newblock V. Isakov and J.-N. Wang,
     \newblock\emph{Increasing stability for determining the potential in the Schr\"odinger equation with attenuation from the Dirichlet-to-Neumann map},
     \newblock Inverse Problems and Imaging, 
     \textbf{8} (2014), 1139-1150.
     
     
\bibitem{L}
     \newblock L. Liang,
     \newblock\emph{Increasing stability for the inverse problem of the Schr\"odinger equation with the partial Cauchy data},
     \newblock Inverse Problems and Imaging, 
     \textbf{9} (2015), 469-478.

\bibitem{M}
     \newblock N. Mandache,
     \newblock \emph{Exponential instability in an inverse problem for the Schr\"odinger equation},
     \newblock Inverse Problems, \textbf{17} (2001), 1435--1444.


\bibitem{NUW}
     \newblock S. Nagayasu, G. Uhlmann and J.-N. Wang,
          \newblock \emph{Increasing stability in an inverse problem for the acoustic equation},
     \newblock Inverse Problems, \textbf{29} (2013), 025012.

\bibitem{OS}
     \newblock P. Ola and E. Somersalo, 
     \newblock \emph{Electromagnetic inverse problems and generalized Sommerfeld potentials},
     \newblock SIAM J. Appl. Math, \textbf{56} (1996), 1129--1145.

\bibitem{OPS}
\newblock P. Ola, L. P\"{a}iv\"{a}rinta, and E. Somersalo,
\newblock\emph{An inverse boundary value problem in electrodynamics},
\newblock Duke Math. J., \textbf{70} (1993), 617--653.




\bibitem{SU1}
     \newblock J. Sylvester and G. Uhlmann, 
     \newblock \emph{A global uniqueness theorem for an inverse boundary value problem},
     \newblock Ann. Math., \textbf{125} (1987), 153--169.



\end{thebibliography}
\end{document}